\documentclass{amsart}
\usepackage{a4wide,amssymb,mathtools,textcomp}

\newcommand{\dom}{\mathrm{dom}}
\newcommand{\IR}{\mathbb R}
\newcommand{\defeq}{\coloneqq}

\newcommand{\F}{\mathcal F}
\newcommand{\w}{\omega}
\newcommand{\e}{\varepsilon}
\newcommand{\Tau}{\mathcal T}
\newcommand{\Stra}{{\$}}
\newcommand{\Strat}{\mbox{\texteuro}}

\newcommand{\diam}{\mathrm{diam}}
\newcommand{\IN}{\mathbb N}

\newtheorem{theorem}{Theorem}[section]
\newtheorem{proposition}[theorem]{Proposition}
\newtheorem{lemma}[theorem]{Lemma}
\newtheorem{claim}[theorem]{Claim}
\newtheorem{corollary}[theorem]{Corollary}

\newtheorem{example}[theorem]{Example}

\theoremstyle{definition}

\newtheorem{definition}{Definition}

\title{Game extensions of floppy graph metrics}
\author{Taras Banakh and Pietro Majer}
\address{T.Banakh: Ivan Franko National University of Lviv (Ukraine), and Jan Kochanowski University in Kielce (Poland)}
\email{t.o.banakh@gmail.com}
\address{P.Majer: Universit\`a di Pisa, Dipartimento di Matematica}

\email{pietro.majer@unipi.it}
\subjclass{05C12; 05C57; 54E35; 91A43}
\keywords{graph metric, full metric, floppy graph metric, metric-extending game}

\begin{document}
\begin{abstract} A {\em graph metric}  on a set $X$ is any function $d: E_d \to\IR_+\defeq\{x\in\IR:x>0\}$ defined on a connected graph $ E_d \subseteq[X]^2\defeq\{A\subseteq X:|A|=2\}$ and such that for every $\{x,y\}\in E_d$ we have $d(\{x,y\})\le\hat d(x,y)\defeq\inf\big\{\sum_{i=1}^nd(\{x_{i-1},x_i\}):\{x,y\}=\{x_0,x_n\}\;\wedge\;\{\{x_{i-1},x_i\}:0<i\le n\}\subseteq E_d \big\}$. A graph metric $d$ is called a {\em full metric} on $X$ if $ E_d =[X]^2$. A graph metric $d: E_d \to\bar\IR_+$ is {\em floppy} if $\hat d(x,y)>\check d(x,y)\defeq \sup\{d(\{a,b\})-\hat d(a,u)-\hat d(b,y):\{a,b\}\in E_d \}$ for every $x,y\in X$ with $\{x,y\}\notin E_d $. We prove that for every floppy graph metric $d: E_d \to\IR_+$ on a set $X$, every  points $x,y\in X$ with $\{x,y\}\notin E_d $, and every real number $r$ with $\frac 13\check d(x,y)+\frac23\hat d(x,y)\le r<\hat d(x,y)$ the function $d\cup\{\langle\{x,y\},r\rangle\}$ is a floppy graph metric. This implies that for every floppy graph metric $d: E_d \to\IR_+$ with countable set $[X]^2\setminus E_d $ and for every indexed family $(F_e)_{e\in[X]^2\setminus  E_d }$ of dense subsets of $\IR_+$, there exists an injective function $r\in\prod_{e\in[X]^2\setminus E_d}F_e$ such that $d\cup r$ is a full metric.  Also, we prove that the latter result does not extend to partial metrics defined on uncountable sets. %On the other hand, on every uncountable set $X$ we construct a partial metric $d$ on $X$ such that for any family $(D_{xy})_{xy\in [X]^2\setminus E_d }$ of nondegenerated subsets of the real line, there exist two distinct pairs $xy,st\in[X]^2\setminus E_d $ and real numbers $r_{xy}\in D_{xy}$ and $r_{st}\in D_{st}$ such that $d\cup\{\langle xy,r_{xy}\rangle,\langle st,r_{st}\rangle\}$ is not a partial metric.
\end{abstract}
\maketitle

%For a topological space $X$ we denote by $\mathcal C(X)$ the set of nonempty compact connected subsets of $X$. In particular, for the half-line $\IR_+\defeq[0,\infty)$ the set $\C(\IR_+)$ is the set of all closed intervals in $\IR_+$.

\section{Introduction}

Problems of extension of metric to metrics to larger domains are classical in Topology and Metric Geometry \cite{BB, BSTZ,Bing,H,Tor,TZ}. Recently this problem was studied in the framework of Metric Graph Theory, see \cite{DMV}, \cite{P}.  
In this paper we explore the problem of extension of a weight function on a graph to a metric that obeys some restrictions on distances between points.  The main result on the existence of such an extension is formulated in terms on an infinite game, called the metric-extending game. This result has been applied in the construction of exotic Banakh spaces in \cite{Banakh}.

We start with recalling some definitions. Let $$\IR_+\defeq\{x\in\IR:x>0\}\quad\mbox{and}\quad\bar\IR_+\defeq\{x\in\IR:x\ge0\}$$be the open and closed half-lines, respectively. For two sets $X,Y$ we write $X\subset Y$ if $X\subseteq Y$ and $X\ne Y$. We denote by $|X|$ the cardinality of a set $X$. For a function $f:X\to Y$ and a set $A$, let $f[A]\defeq\{f(a):a\in A\cap X\}$.  Every function $f:X\to Y$ is identified with its graph $\{\langle x,y\rangle\in X\times Y:y=f(x)\}$. After such indentification, we can operate with functions as with sets, i.e., take unions of functions, intersections, etc. 

\begin{definition} A {\em pseudometric} on a set $X$ is any function $d:X\times X\to\IR$ satisfying the axioms:
\begin{enumerate}
\item $d(x,x)=0$
\item $d(x,y)=d(y,x)$
\item $d(x,z)\le d(x,y)+d(y,z)$
\end{enumerate}
for every $x,y,z\in X$.
A pseudometric $d:X\times X\to\IR$ is a {\em metric} if $d(x,y)>0$ for any distinct points $x,y\in X$. 
\end{definition}

A {\em metric space} is a set $X$ endowed with a metric $d:X\times X\to\IR$. A metric space $X$ is 
\begin{itemize}
\item {\em discrete} if for every $x\in X$ there exists $\e\in\IR_+$ such that $\{y\in X:d(x,y)<\e\}=\{x\}$;
\item {\em separable} if there exists a countable subset $A\subseteq X$ such that for every $x\in X$ and $\e\in\IR_+$,  there exists a point $a\in A$ with $d(a,x)<\e$.
\end{itemize}
It is well-known (and easy to see) that a discrete metric space is separable if and only if it is countable.
\smallskip

A metric on a set $X$ can be also thought as a function on the complete graph $$[X]^2\defeq\{A\subseteq X:|A|=2\}.$$ For two elements $x,y\in X$ it will be convenient to denote the set $\{x,y\}$ by $xy$. So, $[X]^2=\{xy:x,y\in X, \;x\ne y\}$.

\begin{definition} A {\em graph} is a set $E$ such that every element $e\in E$ is a set of cardinality $|e|=2$. Elements of a graph $E$ are called {\em edges} of $E$. The set $V\defeq\bigcup E$ is called the set of {\em vertices} of the graph. A graph $E$ is {\em connected} if for any vertices $x,y\in V$ there exists a sequence $x_0,\dots,x_n\in V$ such that $xy=x_0x_n$ and $\{x_{i-1}x_i:0<i\le n\}\subseteq E$.
\end{definition}

%Observe that $X=\bigcup E$ for every connected graph $E$ on $X$.

\begin{definition} A {\em graph pseudometric} is a function $d:E_d\to\bar \IR_+$ defined on a connected graph $E_d$ such that  for every vertices $x_0,\dots,x_n$ of $E_d$ with $\{x_0x_n\}\cup\{x_{i-1}x_i:0<i\le n\}\subseteq E_d$ the {\em polygonal inequality} $d(x_0x_n)\le\sum_{i=1}^nd(x_{i-1}x_i)$ holds. 
For a  graph pseudometric $d:E_d\to\bar\IR_+$ we denote by $V_d$ the set $\bigcup E_d$ of vertices of the graph $E_d$.
\vskip3pt

A graph pseudometric $d$ is called 
\begin{itemize}
\item a {\em graph metric} if $d(xy)>0$ for every edge $xy\in E_d$;
\item a {\em full pseudometric} if $E_d=[V_d]^2$;
\item a {\em full metric} if $E_d=[V_d]^2$ and $d(x,y)>0$ for every $xy\in  E_d$.
\end{itemize}
A {\em graph pseudometric on a set} $X$ is a graph pseudometric $d$ with $V_d =X$.
\end{definition}

Every full (pseudo)metric $d:E_d\to\bar\IR_+$  determines a (pseudo)metric $$D:V_d\times V_d\to\IR,\quad D(x,y)\defeq\begin{cases}d(xy)&\mbox{if $x\ne y$};\\
0&\mbox{otherwise}.
\end{cases}
$$
Conversely, every (pseudo)metric $D:X\times X\to\IR$ on a set $X$ determines the full (pseudo)metric $d:[X]^2\to\IR_+$, $d(xy)=D(x,y)$, on the complete graph $[X]^2$. So, these two notions are essentially equivalent.

Any function $d: E_d \to\bar\IR_+$ on a connected graph $ E_d $ determines the  {\em shortest-path pseudometric}
$${\textstyle \hat d:V_d\times V_d\to\IR},\quad
\hat d(x,y)\defeq\inf\Big\{\sum_{i=1}^n d(x_{i-1}x_i):x_0=x\wedge x_n=y\wedge \{x_{i-1}x_i:0<i\le n\}\subseteq E_d \Big\}
$$
on the set $V_d=\bigcup E_d$ of vertices of the graph $E_d$. It is easy to see that $\hat d(x,y)\le d(xy)$ for every vertices $x,y\in V_d$ with $xy\in E_d $, and  $d$ is a graph pseudometric if and only if $d(xy)=\hat d(x,y)$ for every $xy\in E_d $, see \cite{DMV}.

The shortest-path pseudometric $\hat d:V_d\times V_d\to\IR$ determines the  pseudometric $$\ddot d:[V_d]^2\times [V_d]^2\to\IR,\quad \ddot d(xy,uv)\defeq\min\{\hat d(x,u)+\hat d(y,v),\hat d(x,v)+\hat d(y,u)\}$$
on the complete graph $[V_d]^2$, see Lemma~\ref{l:double}.

We say that a graph pseudometric $d$ {\em extends} a graph pseudometric $p$ if $p\subseteq d$ and $V_p=V_d$.

The values of possible extensions of a graph pseudometric $d$ are bounded from below by the function $\check d:V_d\times V_d\to\bar\IR_+$ defined by $$
\check d(x,y)\defeq\sup_{ab\in E_d}\max\{0,d(ab)-\hat d(a,x)-\hat d(y,b)\}=\sup_{ab\in E_d }\max\{0,d(ab)-\ddot d(ab,xy)\}$$
for $x,y\in X$, see \cite{P}. 
In fact, the function $\check d:V_d\times V_d\to\bar\IR_+$ is well-defined for any function $d: E_d \to\bar\IR_+$ on a connected graph $ E_d $.

\begin{definition}
A graph pseudometric $d$ is defined to be {\em floppy} if $\check d(x,y)<\hat d(x,y)$ for every vertices $x,y\in V_d$ with $xy\notin E_d$. 
\end{definition}

This definition implies that every full pseudometric is floppy. By \cite{P}, a graph pseudometric $d$ has a unique extensions to a full pseudometric if and only if $\hat d=\check d$.

\begin{example} On the full Cantor tree $2^{\le\w}=\bigcup_{\alpha\le\w}\{0,1\}^\alpha$ consider the partial metric $$d\defeq\{\langle st,2^{-|s|}-2^{-|t|}\rangle:s,t\in 2^{\le \w},\;s\subset t\},$$defined on the set $ E_d =\{st:s,t\in 2^{\le\w}:s\subset t\}$. Then $$\hat d(s,t)=2^{-|s|}+2^{-|t|}-2^{-|s\cap t|+1}\quad\mbox{and}\quad \check d(s,t)=|2^{-|s|}-2^{-|t|}|$$for any $s,t\in 2^{\le\w}$, which implies that the graph metric $d$ on $2^{\le\w}$ is floppy. Moreover, the metric space $([2^{\le\w}]^2,\ddot d)$ is separable and the sets $E_d$ and $[2^{\le \w}]^2\setminus E_d $ have cardinality of continuum.
\end{example}

Many examples of floppy graph (pseudo)metrics are supplied by the following example,  elaborated in Section~\ref{s:glue}.

\begin{example}\label{ex:glue} Let $p$ be a graph pseudometric and $\mathcal F$ be a family of graph pseudometrics satisfying the following conditions:
\begin{enumerate}
\item for every $g\in\F$, the intersection $V_g\cap V_p$ is not empty  and $\hat g(xy)=\hat p(xy)$ for all $xy\in [V_g\cap V_p]^2$;
\item for any distinct $f,g\in\mathcal F$, we have $V_f\cap V_g\subseteq V_p$.
\end{enumerate}
Then the set $d\defeq p\cup\bigcup\F$ is a graph pseudometric on the set $V_d=V_p\cup\bigcup_{f\in \F}V_f$. Moreover, for every $x,y\in V_d$, the shortest-path distance $\hat d(x,y)$ between $x$ and $y$ can be calculated by the formula
$$\hat d(x,y)=\begin{cases}\hat f(x,y)&\mbox{if $x,y\in V_f$ for some $f\in\{p\}\cup \F$};\\
{\displaystyle\inf_{\substack{a\in V_f\cap V_p\\ b\in V_g\cap V_p}}\hat f(x,a)+\hat p(a,b)+\hat g(b,y)}&\mbox{if $x\in V_f$ and $y\in V_g$ for distinct $f,g\in \{p\}\cup\F$}.
\end{cases}
$$
The graph pseudometric $d$ is floppy if $p$ is a full pseudometric, every graph pseudometric $f\in\F$ is floppy, and for every $x\in V_f\setminus V_p$ and $y\in V_p\setminus V_f$, the numbers
 $$\inf\{\hat f(a,x)+\hat f(x,b)-\hat f(ab):a,b\in V_f\cap V_p\}\mbox{ \ and \ }\inf\{p(a,y)+p(y,b)-p(ab):a,b\in V_p\cap V_f\}$$
are  positive. \hfill $\square$
\end{example}

Every graph metric $d$ can be extended to the floppy graph metric
$$\bar d\defeq \{\langle xy,r\rangle:xy\in[V_d]^2\;\wedge\;\check d(x,y)=r=\hat d(x,y)>0\},$$
called the {\em minimal floppy extension} of $d$. The minimal floppy extension $\bar d$ is contained in any floppy graph  metric that extends the graph metric $d$. %Simple examples show that for a graph metric $d$, its minimal floppy extension $\bar d$ does not need to be a graph metric (so can have $\bar d(xy)=0$ for some 2-element edge $xy$). 

The main technical result of the paper is the following theorem on one-step extensions of floppy graph metrics.

\begin{theorem}\label{t:step} For every floppy graph metric $d$, every doubleton $xy\in[V_d]^2\setminus E_d $, and every real number $r$ with $\frac13\check d(x,y)+\frac23\hat d(x,y)\le r<\hat d(x,y)$, the function $D=d\cup\{\langle xy,r\rangle\}$ is a floppy graph metric.
 \end{theorem}
 
Theorem~\ref{t:step} can be used for constructing extensions of floppy graph metrics to full metrics obeying some restrictions on distances between vertices.  It will be convenient to formulate the extension result in the terms of the following game.

Given a graph metric $d:E_d\to\bar\IR_+$, an ordinal $\lambda$, and a family $\F$ of nonempty subsets of the real line,  we define a {\em metric-extending game} $\Game_d(\lambda,\F)$ of length $\lambda$ played by two players, I and II, as follows. The game $\Game_d(\lambda,\F)$ is started by the player I who selects a doubleton $x_0y_0\in [V_d]^2$ and a set $F_0\in\F$. The player II answers selecting a real number $r_0\in F_0$. At the $\alpha$th inning for $\alpha\in\lambda$, the player I selects a doubleton $x_\alpha y_\alpha\in [V_d]^2$ and a set $F_\alpha\in\F$, and the player II answers selecting a real number $r_\alpha\in F_\alpha$. At the end of the game, the player I is declared the winner if the relation $$D\defeq d\cup\{\langle x_\alpha y_\alpha, r_\alpha\rangle:\alpha\in\lambda\}$$ is a full metric. In the other case the player II wins the game.

The metric-extending game $\Game_d(\lambda,\F)$ is nontrivial if the family $\F$ consists of nondegenerated sets, so the second player does have some a choice at his/her moves. A set is {\em nondegenerated} if it contains more than one element.

Theorem~\ref{t:step} implies the following theorem on game extensions of floppy graph pseudometrics.

\begin{theorem}\label{t:main} Let $d: E_d \to\bar\IR_+$ be a floppy graph metric and $\Tau$ be the family of nonempty open sets in $\IR_+$. If the cardinal $\lambda=|[V_d]^2\setminus E_d |$ is at most countable, then the player I has a winning strategy in the metric-extending game $\Game_d(\lambda,\Tau)$.
\end{theorem}

The following theorem shows that in general, a winning strategy of the first player in the metric-extending game does depend on the moves of the second player.

\begin{theorem}\label{t:main2} Let $d: E_d \to\bar\IR_+$ be a floppy graph metric. If the subspace $[V_d]^2\setminus E_d $ of the metric space $([V_d]^2,\ddot d)$ is not discrete, then for every family $(F_{xy})_{xy\in[V_d]^2\setminus E_d }$ of non-degenerated subsets of the real line, there exists an function $r\in \prod_{xy\in [V_d]^2\setminus E_d }F_{xy}$ such that the set $d\cup r$ is not a metric.
\end{theorem}

Our final theorem shows that the countability of $[V_d]^2\setminus  E_d $ is essential in Theorem~\ref{t:main}.

\begin{theorem}\label{t:main3} Let $d: E_d \to\bar\IR_+$ be a graph metric  such that the metric space $([V_d]^2,\ddot d)$ contains a separable uncountable subspace $S\subseteq[V_d]^2\setminus E_d $. Then for every ordinal $\lambda$ and every family $\F$ of non-degenerated subsets of the real line, the player II has a winning strategy in the metric-extending game $\Game_d(\lambda,\F)$.
\end{theorem}

Theorems~\ref{t:step}, \ref{t:main}, \ref{t:main2}, \ref{t:main3} will be proved in Sections~\ref{s:lemma}, \ref{s:main}, \ref{s:main2}, \ref{s:main3}, respectively. Theorem~\ref{t:main} implies the following corollary that answers the question \cite{MO} asked by the first author at {\tt MathOverflow}.

\begin{corollary}  Let $p: E_p \to\bar\IR_+$ be a floppy graph metric such that $|[V_p]^2\setminus E_p |\le\w$. For any indexed family $(F_{e})_{e\in[V_p]^2\setminus E_p }$ of dense sets in $\bar \IR_+$, there exists an injective function $r\in\prod_{e\in[V_p]^2\setminus E_p}F_e$ such that $p\cup r$ is a full metric. 
\end{corollary}

\begin{proof} Since the cardinal $\lambda=|[V_p]^2\setminus E_p|$ is at most countable, by Theorem~\ref{t:main}, the player I has a winning strategy  in the metric-extending game $\Game_p(\lambda,\Tau)$. According to this strategy, at the $n$-th inning the player $I$ selects a doubleton $x_ny_n\in [V_p]^2\setminus E_p$ and a nonempty open set $U_n\subseteq\IR_+$. Then the player II selects a real number $r_n\in U\cap F_{x_ny_n}\setminus\{r_k:k<n\}$. Since the strategy of the player I is winning, the set $d\defeq p\cup\{\langle x_ny_n,r_n\rangle:n\in\lambda\}$ is a full metric. It follows that the set $r=\{\langle x_ny_n,r_n\rangle:n\in\lambda,\;x_ny_n\notin E_p\}=d\setminus p$ is a required injective function with $r\in \prod_{e\in [V_p]^2\setminus E_p}$ and $d=p\cup r$.
\end{proof}
 %because of the following example, constructed in Section~\ref{s:ex}.

%\begin{example}\label{ex} For any uncountable set $X$ there exists a partial floppy metric $d:\dom[X]\to\IR_+$ such that for every indexed family $(D_{xy})_{xy\in[X]^2\setminus E_d }$ of nondegenerated sets in the real line there exist two distinct pairs $xy,st\in [X]\setminus E_d $ and two real numbers $r_{xy}\in D_{xy}$ and $r_{st}\in D_{st}$ such that the function $d\cup\{\langle xy,r_{xy}\rangle,\langle st,r_{st}\rangle\}$ is not a partial metric.
%\end{example}

%A set is {\em nondegenerated} if it contains more than one element.

%Finally we present an example of a (floppy) partial metric that often appears in ``practice'', see e.g. \cite{Ban}.

 \section{The distance between doubletons}\label{s:double}
 
In this section we show that for any function $d: E_d \to\bar\IR_+$ on a connected group $ E_d \subseteq[X]^2$, the function
$$\ddot d:[X]^2\times [X]^2\to\bar\IR_+,\quad\ddot d(xy,uv)\defeq\min\{\hat d(x,u)+\hat d(v,y),\hat d(x,v)+\hat d(u,y)\},$$is a pseudometric on the complete graph $[E_d]^2$. In the definition of $\ddot d$, $\hat d:X\times X\to\bar\IR_+$ is the pseudometric on $X$, defined by 
$$\hat d(x,y)\defeq\inf\Big\{\sum_{i=1}^nd(x_{i-1}x_i):x_0x_n=xy\mbox{ and }\{x_{i-1}x_i:0<i\le n\}\subseteq E_d \Big\}.$$
 
\begin{lemma}\label{l:double} For any  points $a,b,u,v,x,y\in X$ we have
$$\ddot d(ab,xy)\le \ddot d(ab,uv)+\ddot d(uv,xy)\quad\mbox{and}\quad \hat d(a,b)\le \hat d(u,v)+\ddot d(ab,uv).$$
\end{lemma}

\begin{proof} By the definition of $\ddot d(ab,uv)$ and $\ddot d(uv,xy)$, there exist  points $\tilde u,\tilde v\in uv$ and $\tilde x,\tilde y\in xy$ such that $\tilde u\tilde v=uv$, $\tilde x\tilde y=xy$, $\ddot d(ab,uv)=\hat d(a,\tilde u)+\hat d(b,\tilde v)$ and $\ddot d(uv,xy)=\hat d(\tilde u,\tilde x)+\hat d(\tilde v,\tilde y)$. Then 
$$\ddot d(ab,uv)\le \hat d(a,\tilde x)+\hat d(b,\tilde y)
\le \hat d(a,\tilde u)+\hat d(\tilde u,\tilde x)+\hat d(b,\tilde v)+\hat d(\tilde v,\tilde y)=\ddot d(ab,uv)+\hat d(uv,xy)
$$
and $$\hat d(a,b)\le \hat d(a,\tilde u)+\hat d(\tilde u,\tilde v)+\hat d(\tilde v,b)=\hat d(a,\tilde u)+\hat d(b,\tilde v)+\hat d(u,v)=\ddot d(ab,uv)+\hat d(u,v).$$
\end{proof}

\section{One step extensions of partial metrics}\label{s:lemma}

The last statement of the following proposition implies Theorem~\ref{t:step}.

\begin{proposition}\label{p:step} For every graph pseudometric $d$ on a set $X$, every doubleton $xy\in[X]^2\setminus E_d $, and every real number $r$ with $\check d(xy)\le r\le \hat d(xy)$, the function $D\defeq d\cup\{\langle xy,r\rangle\}$ is a graph pseudometric such that for every $u,v\in X$ the following conditions are satisfied:
\begin{enumerate}
\item $\hat D(u,v)\le \hat d(u,v)$ and $\max\{\check d(u,v),r-\ddot d(xy,uv)\}\le \check D(u,v)$.
\item If $\hat d(u,v)\ne\hat D(u,v)$, then $\hat d(u,v)-(\hat d(x,y)-r)\le \hat D(u,v)=r+\ddot d(xy,uv)$.
\item If $\hat d(u,v)\ne\hat D(u,v)$, then $\hat D(u,v)-\check d(u,v)\ge r-\check d(x,y)$.
%\item If $\check d(u,v)\ne \check D(u,v)>\hat d(x,y)-2r$, then $\hat d(u,v)-\check D(u,v)\ge \hat d(x,y)-r$;
\item If $\check d(u,v)\ne\check D(u,v)\ne r-\ddot d(xy,uv)$ and $\check D(u,v)>\hat d(x,y)-2r$, then\newline
$\check D(u,v)-\check d(u,v)\le\hat d(x,y)-r$ and $\hat d(u,v)-\check D(u,v)\ge r-\check d(x,y).$
\item If $\frac13\check d(x,y)+\frac23\hat d(x,y)\le r\le\hat d(x,y)$, then\newline $\hat D(u,v)-\check D(u,v)\ge\min\{\hat d(u,v)-\check d(u,v),\hat d(x,y)-r,2\ddot d(xy,uv)\}$.
\end{enumerate}
\end{proposition}

\begin{proof} Assuming that $D$ is not a graph pseudometric, we can find a sequence of pairwise distinct points $x_0,x_1,\dots,x_n\in X$ such that $\{x_0,x_n\}\cup\{x_{i-1}x_i:0<i\le n\}\subseteq\dom[D]$ and $\sum_{i=1}^nD(x_{i-1}y_{i-1})<D(x_0x_n)$. 

If $x_0x_n=xy$, then the strict inequality $\sum_{i=1}^nD(x_{i-1}y_{i-1})<D(x_0x_n)=D(xy)$ implies that $x_{i-1}x_i\ne xy$ and hence $D(x_{i-1}x_i)=d(x_{i-1}x_i)$ for all $i\in\{1,\dots,n\}$. Then $\hat d(x,y)\le\sum_{i=1}^nd(x_{i-1}x_i)=\sum_{i=1}^nD(x_{i-1}x_i)<D(xy)=r\le \hat d(x,y)$, which is a  contradiction showing that $x_0x_n\ne xy$ and hence $x_0x_n\in\dom[D]\setminus\{xy\}= E_d $ and $D(x_0x_n)=d(x_0x_n)$. The strict inequality $\sum_{i=1}^nD(x_{i-1},x_i)<D(x_0x_n)=d(x_0x_n)\le\sum_{i=1}^nd(x_{i-1}x_i)$ ensures that $x_{k-1}x_k=xy$ for some $k\in\{1,\dots,n\}$. It follows from $\{x_i:i\in\{0,\dots,n\}\setminus\{k-1,k\}\}\cap \{x_{k-1},x_k\}=\emptyset$ that $x_{i-1}x_i\in E_d $ for every $i\in\{1,\dots,n\}\setminus\{k\}$ and hence 
\begin{multline*}\check d(x,y)\le r=D(xy)=D(x_{k-1}x_k)<D(x_0x_n)-\sum_{i=1}^{k-1}D(x_{i-1}x_i)-\sum_{i=k+1}^nD(x_{i-1}x_i)\\=d(x_0x_n)-\sum_{i=1}^{k-1}d(x_{i-1}x_i)-\sum_{i=k+1}^nd(x_{i-1}x_i)\le d(x_0x_n)-\hat d(x_0,x_{k-1})-\hat d(x_k,x_n)\le \check d(x,y),
\end{multline*}
which is a desired contradiction showing that $D$ is a graph pseudometric.
\smallskip

It remains to prove the statements (1)--(5). The proof of the statement (1) is divided into two lemmas.

\begin{lemma}\label{cl:1} For every elements $u,v\in X$, we have $\hat D(u,v)\le \hat d(u,v)$.
\end{lemma}

\begin{proof} Assuming that $\hat d(u,v)<\hat D(u,v)$, we can find a sequence $x_0,\dots,x_n\in X$ such that $x_0x_n=uv$, $\{x_{i-1}x_i:0<i\le n\}\subseteq  E_d \subseteq \dom[D]$ and $\sum_{i=1}^nd(x_{i-1}x_i)<\hat D(u,v)=\hat D(x_0,x_n)$. Then $$\hat D(x_0,x_n)\le\sum_{i=1}^nD(x_{i-1}x_i)=\sum_{i=1}^nd(x_{i-1}x_i)<\hat D(x_0,x_n)$$is a desired contradiction showing that $\hat D(u,v)\le \hat d(u,v)$.
\end{proof}

\begin{lemma}\label{cl:2} For every elements $u,v\in X$ we have $r-\ddot d(uv,xy)\le \check D(u,v)$ and $\check d(u,v)\le \check D(u,v)$.
\end{lemma}

\begin{proof} The definition of $\check D(u,v)>0$ and Lemma~\ref{cl:1} ensure that
$$
\begin{aligned}
\check D(u,v)&\ge \max\{D(xy)-\hat D(x,u)-\hat D(v,y),D(xy)-\hat D(x,v)-\hat D(u,y)\}\\&=r-\min\{\hat D(x,u)+\hat D(v,y),\hat D(x,v)+\hat D(v,y)\}\\
&\ge r-\min\{\hat d(x,u)+\hat d(v,y),\hat d(x,v)+\hat d(v,y)\}=r-\ddot d(xy,uv).
\end{aligned}
$$

 Assuming that $\check D(u,v)<\check d(u,v)$, we can find a doubleton $ab\in E_d \subseteq\dom[D]$ such that $d(a,b)-\hat d(a,u)-\hat d(v,b)>\check D(u,v)$. Lemma~\ref{cl:1} implies that $\hat D(a,u)\le \hat d(a,u)$ and $\hat D(v,b)\le \hat d(v,b)$ and hence $$\check D(u,v)<d(ab)-\hat d(a,u)-\hat d(v,b)\le D(a,b)-\hat D(a,u)-\hat D(v,b)\le \check D(u,v),$$
which is a contradiction showing that $\check D(u,v)\ge \check d(u,v)$.
\end{proof}

The statement (2) of Proposition~\ref{p:step} is proved in the following lemma.

\begin{lemma}\label{cl:3} For every $u,v\in X$ with $\hat d(u,v)\ne\hat D(u,v)$ we have
$$\hat D(u,v)=r+\ddot d(xy,uv)\ge \hat d(u,v)-(\hat d(x,y)-r).$$%\min\{\hat d(u,x)+\hat d(y,v),\hat d(u,y)+\hat d(x,v)\}$.
\end{lemma}

\begin{proof} The triangle inequality for the full pseudometric $\hat D$ and Lemma~\ref{cl:1} imply $$\hat D(u,v)\le \hat D(u,x)+D(x,y)+\hat D
(y,v)\le \hat d(u,x)+r+\hat d(y,v).$$ By analogy we can show that $\hat D(u,v)\le\hat d(u ,y)+r+\hat d(x,v)$. Therefore, $$\hat D(u,v)\le r+ \min\{\hat d(u,x)+\hat d(y,v),\hat d(u,y)+\hat d(x,v)\}=r+\ddot d(xy,uv).$$
By Lemma~\ref{cl:1}, the inequality $\hat D(u,v)\ne \hat d(u,v)$ implies the strict  inequality $\hat D(u,v)<\hat d(u,v)$. Assuming that  $\hat D(u,v)\ne r+\ddot d(uv,xy)$, we can choose $\e>0$ such that $$\e<\min\{\hat d(u,v)-\hat D(u,v),r+\ddot d(uv,xy)-\hat D(u,v)\}.$$ By the definition of $\hat D(u,v)$, there exists a sequence of pairwise distinct points $x_0,\dots,x_n\in X$ such that $x_0=u$, $x_n=v$, $\{x_{i-1}x_i:0<i\le n\}\subseteq\dom[D]$ and $\sum_{i=1}^nD(x_{i-1}x_i)<\hat D(u,v)+\e$. 

Assuming that $x_{k-1}x_k\ne xy$ for all $k\in\{1,\dots,n\}$, we conclude that $\{x_{i-1}x_i:0<i\le n\}\subseteq\dom[D]\setminus\{xy\}= E_d $ and hence $$\hat d(u,v)\le\sum_{i=1}^nd(x_{i-1}x_i)=\sum_{i=1}^nD(x_{i-1}x_i)<\hat D(u,v)+\e<\hat d(u,v),$$which is a contradiction showing that $x_{k-1}x=xy$ for some $k\in\{1,\dots,n\}$. 

Then $\{x_{i-1}x_i:i\in\{1,\dots,n\}\setminus\{k\}\}\subseteq\dom[D]\setminus\{xy\}= E_d $ and hence 
$$\sum_{i=1}^nD(x_{i-1}x_i)=\sum_{i=1}^{k-1}d(x_{i-1}x_i)+r+\sum_{i=k+1}^nd(x_{i-1}x_i)\ge\hat d(x_0,x_{k-1})+r+\hat d(x_k,x_n).$$
The choice of $\e$ ensures that
\begin{multline*}
\ddot d(uv,xy)+r=\ddot d(x_0x_n,xy)+r\le \hat d(x_0,x_{k-1})+\hat d(x_k,x_n)+r\\
\le \sum_{i=1}^nD(x_{i-1}x_i)<\hat D(u,v)+\e<\ddot d(uv,xy)+r,
\end{multline*}
which is a contradiction completing the proof of the equality $\hat D(u,v)=\ddot d(uv,xy)+r$. 

Lemma~\ref{l:double} implies
$$\hat d(u,v)\le \ddot d(uv,xy)+\hat d(x,y)=\ddot d(uv,xy)+r+(\hat d(x,y)-r)=\hat D(u,v)+(\hat d(x,y)-r).$$
\end{proof}

Now fix any elements $u,v\in X$. 
The statement (3) of Proposition~\ref{p:step} is proved in the following lemma.

\begin{lemma}\label{cl:5} If $\hat d(u,v)\ne\hat D(u,v)$, then $\hat D(u,v)-\check d(u,v)\ge r-\check d(x,y)$.
\end{lemma}

\begin{proof}  By Claim~\ref{cl:3}, $\hat D(u,v)=r+\ddot d(uv,xy)$. Assuming that  $\hat D(u,v)-\check d(u,v)<r-\check d(x,y)$, we conclude that $\check d(u,v)>\hat D(u,v)-r+\check d(x,y)=\ddot d(xy,uv)+\check d(x,y)\ge 0$. By the definition of the number $\check d(u,v)>\ddot d(xy,uv)+\check d(x,y)$, there exist points $a,b\in X$ such that $ab\in E_d $ and $d(ab)-\ddot d(ab,uv)>\ddot d(xy,uv)+\check d(x,y)$. Combining this inequality with the triangle inequality for the pseudometrics $\hat d$ and $\ddot d$, we obtain a desired contradiction:
$$\hat d(x,y)\ge d(ab)-\ddot d(ab,xy)\ge d(ab)-\ddot d(ab,uv)-\ddot d(uv,xy)>\ddot d(xy,uv)+\hat d(x,y)-\ddot d(uv,xy)=\hat d(x,y).$$
\end{proof}

The statement (4) of Proposition~\ref{p:step} is proved in the following lemma.

%\begin{lemma}\label{cl:4} If $\check d(u,v)\ne \check D(u,v)>\hat d(x,y)-2r$, then $\hat d(u,v)-\check D(u,v)\ge \hat d(x,y)-r$.
%\end{lemma}

%The statement (5) of Proposition~\ref{p:step} is proved in the following lemma.

\begin{lemma}\label{cl:6} If $\check d(u,v)\ne\check D(u,v)\ne r-\ddot d(xy,uv)$ and $\check D(u,v)>\hat d(x,y)-2r$, then
$$\check D(u,v)-\check d(u,v)\le\hat d(x,y)-r\quad\mbox{and}\quad \hat d(u,v)-\check D(u,v)\ge r-\check d(x,y).$$
\end{lemma}

\begin{proof} Assume that $\check d(u,v)\ne \check D(u,v)\ne r-\ddot d(xy,uv)$ and $\check D(u,v)>\hat d(x,y)-2r$. Then $\check D(u,v)<\min\{\check d(u,v),r-\ddot d(xy,uv)\}$, by Lemma~\ref{cl:3}. Choose a positive number $$\e<\min\{\check D(u,v)-\check d(u,v),\check D(u,v)-r+\ddot d(xy,uv),\check D(u,v)-\hat d(x,y)+2r\}$$having the following two properties:
\begin{itemize}
\item[(i)] if $\check D(u,v)-\check d(u,v)>\hat d(x,y)-r$, then $\check D(u,v)-\check d(u,v)-\e>\hat d(x,y)-r$;
\item[(ii)] if $\hat d(u,v)-\check D(u,v)<r-\check d(x,y)$, then $\hat d(u,v)-\check D(u,v)+\e<r-\check d(x,y)$.
\end{itemize}

By the definition of $\check D(u,v)>\check d(u,v)\ge 0$, there exist points $a,b\in X$ such that $ab\in\dom[D]$ and $D(ab)-\hat D(a,u)-\hat D(v,b)>\check D(u,v)-\e$.

We claim that $ab\ne xy$. Indeed, assuming that $ab=xy$, we obtain that
$\hat D(a,u)+\hat D(v,b)<D(ab)-(\check D(u,v)-\e))\le r-\check d(u,v)\le r$. 
%$$
%\begin{aligned}
%\check D(u,v)\ge \max\{D(xy)-\hat D(x,u)-\hat D(v,y),D(yx)-\hat D(x,v)-\hat D(v,y)\}\\
%=r-\hat D(xy,uv)\ge r-\hat D(a,u)-\hat D(v,b)>\check D(u,v)-\e\ge \check d(u,v)\ge 0
%\end{aligned}
%$$
Applying Lemma~\ref{cl:3}, we conclude that 
$\hat D(a,u)+\hat D(v,b)=\hat d(a,u)+\hat d(v,b)$. If $\hat D(a,v)+\hat D(u,b)\le \hat D(a,u)+\hat D(v,b)$, then $\hat D(a,v)+\hat D(u,b)\le \hat D(a,u)+\hat D(v,b)<r$ and Lemma~\ref{cl:3} ensures that $\hat D(a,v)+\hat D(u,b)=\hat d(a,v)+\hat d(u,b)$.
In this case $\ddot D(xy,uv)=\min\{\hat D(a,u)+\hat D(v,b),\hat D(a,v)+\hat D(u,b)\}=
\min\{\hat d(a,u)+\hat d(v,b),\hat d(a,v)+\hat d(u,b)\}=\ddot d(xy,uv)$.
If $\hat D(a,v)+\hat D(u,b)>\hat D(a,u)+\hat D(v,b)$, then by Lemma~\ref{cl:1},
$$\hat d(a,u)+\hat d(v,b)=\hat D(a,u)+\hat D(v,b)<\hat D(a,v)+\hat D(u,b)\le \hat d(a,v)+\hat d(u,b)$$ and
$$
\begin{aligned}
\ddot d(xy,uv)&=\min\{\hat d(a,u)+\hat d(v,b),\hat d(a,v)+\hat d(u,b)\}=\hat d(a,u)+\hat d(v,b)=\hat D(a,u)+\hat D(v,b)\\
&=\min\{\hat D(a,u)+\hat D(v,b),\hat D(a,v)+\hat D(u,b)\}=\ddot D(ab,uv)=\ddot D(xy,uv).
\end{aligned}
$$
In both cases we obtain $\ddot d(xy,uv)=\ddot D(xy,uv)=\ddot D(ab,uv)$.
Then 
$$r-\ddot d(xy,uv)=r-\ddot D(ab,uv)\ge D(ab)-\hat D(a,u)-\hat D(v,b)>\check D(u,v)-\e>r-\ddot d(xy,uv),$$which is a contradiction showing that $ab\ne xy$ and hence $ab\in\dom[D]\setminus\{xy\}= E_d $.

Assuming that $\hat D(a,u)+\hat D(v,b)=\hat d(a,u)+\hat d(v,b)$, we conclude that
$$\check d(u,v)\ge d(ab)-\hat d(a,u)-\hat d(v,b)=D(ab)-\hat D(a,u)-\hat D(v,b)>\check D(u,v)-\e\ge\check d(u,v),$$which is a contradiction showing that $\hat D(a,u)+\hat D(v,b)\ne \hat d(a,u)+\hat d(v,b)$ and hence $\hat D(a,u)+\hat D(v,b)<\hat d(a,u)+\hat d(v,b)$, by Lemma~\ref{cl:1}.

Assuming that $\hat D(a,u)<\hat d(a,u)$ and $\hat D(v,b)<\hat d(v,b)$ and applying Lemma~\ref{cl:3}, we conclude that $\hat D(a,u)=r+\ddot d(au,xy)$ and $\hat D(v,b)=r+\ddot d(vb,xy)$. By the definition of $\ddot d(au,xy)$ and $\ddot d(bv,xy)$, there exist elements $x',y'\in \{x,y\}$ such that $\hat d(a,x')\le \ddot d(au,xy)$ and $\hat d(y',b)\le\ddot d(xy,vb)$. Then
$$
\begin{aligned}
\check D(u,v)-\e&<D(ab)-\hat D(a,u)-\hat D(v,b)=d(ab)-\ddot d(au,xy)-r-\ddot d(xy,vb)-r\\
&\le d(ab)-\hat d(a,x')-\hat d(y',b)-2r\le \hat d(x',y')-2r\le \hat d(x,y)-2r\le\check D(u,v)-\e,
\end{aligned}
$$
which is a contradiction showing that either $\hat D(a,u)<\hat d(a,u)$ and $\hat D(v,b)=\hat d(v,b)$ or $\hat D(a,u)=\hat d(a,u)$ and $\hat D(v,b)<\hat d(v,b)$.

First assume that  $\hat D(a,u)<\hat d(a,u)$ and $\hat D(v,b)=\hat d(v,b)$. By Lemma~\ref{cl:3}, $\hat D(a,u)=r+\ddot d(au,xy)$. Applying Lemma~\ref{l:double}, we obtain 
$$
\begin{aligned}
\check d(u,v)&\ge d(ab)-\hat d(a,u)-\hat d(v,b)\ge d(ab)-(\ddot d(au,xy)+\hat d(x,y))-\hat d(v,b)\\
&=D(ab)-\hat D(a,u)+r-\hat d(x,y)-\hat D(v,b)>\check D(u,v)-\e+r-\hat d(x,y)
\end{aligned}
$$
and hence $\check D(u,v)-\check d(u,v)\le \hat d(x,y)-r$ by the condition (i) in the choice of $\e$.

On the the other hand,
$$\check d(x,y)\ge d(ab)-\hat d(a,x)-\hat d(y,b)\ge d(ab)-\hat d(a,x)-\hat d(y,u)-\hat d(u,b)$$
and 
$$\check d(x,y)=\check d(y,x)\ge d(ab)-\hat d(a,y)-\hat d(x,b)\ge d(ab)-\hat d(a,y)-\hat d(x,u)-\hat d(u,b)$$
imply
$$
\begin{aligned}
\check d(x,y)&\ge\max\{d(ab)-\hat d(a,x)-\hat d(y,u)-\hat d(u,b),
d(a,b)-\hat d(a,y)-\hat d(x,u)-\hat d(u,b)\}\\
&=d(ab)-\min\{\hat d(a,x)+\hat d(y,u),\hat d(a,y)+\hat d(x,u)\}-\hat d(u,b)\\
&\ge d(ab)-\ddot d(au,xy)-\hat d(u,v)-\hat d(v,b)=D(ab)-\hat D(a,u)+r-\hat d(u,v)-\hat D(v,b)\\
&>\check D(u,v)-\e+r-\hat d(u,v)
\end{aligned}
$$
and $$\hat d(u,v)-\check D(u,v)\ge r-\check d(x,y),$$
by the condition (ii) in the choice of $\e$.

By analogy we can prove the inequalities 
$\check D(u,v)-\check d(u,v)\le\hat d(x,y)-r$ and $\hat d(u,v)-\check D(u,v)\ge r-\check d(x,y)$ in case $\hat D(a,u)=\hat d(a,u)$ and $\hat D(v,b)<\hat d(v,b)$.
\end{proof}

The final statement (5) of Proposition~\ref{p:step} is proved in our final lemma.

\begin{lemma}\label{l:7} If $\frac13\check d(x,y)+\frac23\hat d(x,y)\le r\le\hat d(x,y)$, then $$\hat D(u,v)-\check D(u,v)\ge\min\{\hat d(u,v)-\check d(u,v),\hat d(x,y)-r,2\ddot d(xy,uv)\}.$$
\end{lemma}

\begin{proof} If $r=\hat d(x,y)$, then 
$$\hat D(u,v)-\check D(u,v)\ge 0=\hat d(x,y)-r=\min\{\hat d(u,v)-\check d(u,v),\hat d(x,y)-r,2\ddot d(xy,uv)\}$$and we are done. So, assume that $r<\hat d(x,y)$ and hence $\check d(x,y)<\hat d(x,y)$. Then 
 $$\hat d(x,y)-2r\le\hat d(x,y)-\tfrac23\check d(x,y)-\tfrac 43\hat d(x,y)<0\le \check D(u,v).$$

To prove the displayed inequality in Lemma~\ref{l:7}, we consider separately 6 cases.
\smallskip

1. If $\hat D(u,v)=\hat d(u,v)$ and $\check D(u,v)=\check d(u,v)$, then
 $$\hat D(u,v)-\check D(u,v)=\hat d(u,v)-\check d(u,v)\ge\min\{\hat d(u,v)-\check d(u,v),\hat d(x,y)-r,2\ddot d(xy,uv)\}.$$

 2. If $\hat D(u,v)\ne \hat d(u,v)$ and $\check D(u,v)=\check d(u,v)$, then 
 $$\hat D(u,v)-\check D(u,v)=\hat D(u,v)-\check d(u,v)\ge r-\check d(x,y)\ge\hat d(x,y)-r,$$
 by Lemma~\ref{cl:5} and the inequality $\frac13\check d(x,y)+\frac23\hat d(x,y)\le r<\hat d(x,y)$.
 \smallskip
 
 3. If $\hat D(u,v)=\hat d(u,v)$ and $\check d(u,v)\ne\check D(u,v)\ne r-\ddot  d(xy,uv)$, then 
 $$\hat D(u,v)-\check D(u,v)=\hat d(u,v)-\check D(u,v)\ge r-\check d(x,y)\ge \hat d(x,y)-r,$$
 by Lemma~\ref{cl:6} and the inequality $\frac13\check d(x,y)+\frac23\hat d(x,y)\le r<\hat d(x,y)$.
 \smallskip

 4. If $\hat D(u,v)=\hat d(u,v)$ and $\check d(u,v)\ne\check D(u,v)=r-\ddot d(xy,uv)$, then by Lemma~\ref{l:double}, $\hat d(x,y)\le \hat d(u,v)+\ddot d(xy,uv)=\hat D(u,v)+r-\check D(u,v)$ and hence
 $\hat D(u,v)-\check D(u,v)\ge \hat d(x,y)-r.$
 \smallskip

5. If $\hat D(u,v)\ne \hat d(u,v)$ and $\check d(u,v)\ne\check D(u,v)=r-\ddot d(xy,uv)$, then 
$$\hat D(u,v)-\check D(u,v)=(r+\ddot d(xy,uv))-(r-\ddot d(xy,uv))=2\ddot d(xy,uv),$$
by Lemma~\ref{cl:3}.
\smallskip

6. Finally assume that $\hat D(u,v)\ne \hat d(u,v)$ and $\check d(u,v)\ne \check D(u,v)\ne r-\ddot d(xy,uv)$.  By Lemma~\ref{cl:6}, 
 $$\hat d(u,v)-\check D(u,v)\ge r-\check d(x,y).$$
 On the other hand, $$\hat d(u,v)-\hat D(u,v)\le \hat d(x,y)-r,$$ by Lemma~\ref{cl:3}. Subtracting these inequalities, we obtain
\begin{multline*}\hat D(u,v)-\check D(u,v)=\hat d(u,v)-\check D(u,v)-(\hat d(u,v)-\hat D(u,v))\ge r-\check d(x,y)-\hat d(x,y)+r\\ \ge\tfrac23\check d(x,y)+\tfrac43\hat d(x,y)-\check d(x,y)-\hat d(x,y)=\tfrac13(\hat d(x,y)-\check d(x,y))\ge \hat d(x,y)-r.
\end{multline*}

\end{proof}
 \end{proof}

\section{Proof of Theorem~\ref{t:main}}\label{s:main}

Let $d$ be a floppy graph metric of a set $X$ and $\Tau$ be the family of all nonempty open sets in $\IR_+$. Assuming that the cardinal $\lambda=|[X^2]\setminus  E_d |$ is at most countable, we shall describe a winning strategy of the player I in the metric-extending game $\Game_d(\lambda,\Tau)$.
First we give a precise definition of a (winning) strategy of the first player in the game $\Game_d(\lambda,\Tau)$.

Let $\IR_+^{<\lambda}\defeq\bigcup_{\alpha\in\kappa}\IR_+^\alpha$. A {\em strategy} of the player I in the game $\Game_d(\lambda,\Tau)$ is a function $\Stra:\IR^{<\lambda}\to [X]^2\times \Tau$. The function $\Stra$ determines unique functions $\Stra_1:\IR_+^{<\lambda}\to[X]^2$ and $\Stra_2:\IR_+^{<\lambda}\to\Tau$ such that $\Stra(r)=\langle\Stra_1(r),\Stra_2(r)\rangle$ for all $r\in \IR_+^{<\lambda}$.

 A strategy $\Stra:\IR_+\to[X]^2\times\Tau$ of the player I is {\em winning} if for every function $s\in\IR_+^\lambda$ with $s(\alpha)\in\Stra_2(r{\restriction}_\alpha)$ for all $\alpha\in\lambda$, the set $d\cup\{\langle \Stra_1(s{\restriction}_\alpha),s(\alpha)\rangle:\alpha\in\kappa\}$ is a full metric on $X$. 
 
Now we are ready to define a wining strategy $\Stra:\IR_+\to[X]^2\times\Tau$ of the player I in the game $\Game_d(\lambda,\Tau)$. Fix any bijective function $f:\lambda\to [X^2]\setminus  E_d $ and for every ordinal $\alpha\in\lambda$, find points $x_\alpha,y_\alpha\in X$ such that $f(\alpha)=x_\alpha y_\alpha\defeq\{x_\alpha,y_\alpha\}$. 
 
 For every ordinal $\beta\in\lambda$ and function $s\in\IR_+^\beta$, consider the set $D_s=d\cup\{\langle f(\alpha),r(\alpha)\rangle:\alpha\in\beta\}$. Choose any nonempty open set $U_s\subseteq\IR_+$ such that $$U_s\subseteq \{r\in\IR_+:\tfrac13\check D_s(x_\alpha,y_\alpha)+\tfrac23\hat D_s(x_\alpha,y_\alpha)<r<\hat D_s(x_\alpha,y_\alpha)\}$$if $D_s$ is a floppy graph metric on $X$.
 
 Define a strategy $\Stra:\IR_+^{<\lambda}\to [X]^2\times\Tau$ assigning to every  function $s\in\IR_+^{<\lambda}$ the pair $\langle x_\alpha y_\alpha,U_s\rangle$ where $\alpha$ is a unique ordinal such that $s\in\IR_+^\alpha$. Let us show that the strategy $\Stra$ is winning. Fix any function $s\in\IR_+^\kappa$ such that $s(\alpha)\in \Stra_2(s{\restriction}_\alpha)$ for every $\alpha\in\kappa$.
For every ordinal $\beta\le\lambda$, consider the set $d_\beta=d\cup\{\langle x_\alpha y_\alpha,s(\alpha)\rangle:\alpha\in\beta\}$. By induction we shall prove that for every ordinal $\beta\in\kappa$, the set $d_\beta$ is a floppy graph metric on $X$.

For $\beta=0$, the function $d_0=d$ is a floppy graph metric on $X$ by the choice of $d$. Assume that for some nonzero ordinal $\beta\in\kappa$ we have proved that that the sets $d_\alpha$, $\alpha\in\beta$, are floppy graph metrics on $X$. Since $\beta<\lambda\le\w$, the ordinal $\beta$ is a successor ordinal and hence $\beta=\alpha+1$ for some ordinal $\alpha$. By the inductive hypothesis, the set $d_\alpha$ is a floppy graph metric on $X$. The definition of $d_\alpha$ ensures that $\dom[d_\alpha]= E_d \cup\{x_\gamma y_\gamma:\gamma\in\alpha\}$ and hence $x_\alpha y_\alpha\notin\dom[d_\alpha]$. Since $$s(\alpha)\in \Stra_2(s{\restriction}_\alpha)=U_{s{\restriction}_\alpha}\subseteq \{r\in\IR_+:\tfrac13\check d_\alpha(x_\alpha,y_\alpha)+\tfrac23\hat d_\alpha(x_\alpha,y_\alpha)<r<\hat d_\alpha(x_\alpha,y_\alpha)\},$$the set $d_\beta=d_\alpha\cup\{\langle x_\alpha y_\alpha,r(\alpha)\rangle\}$ is a floppy graph metric on $X$, by Theorem~\ref{t:step}.

After completing the inductive construction, observe that $d_\lambda=\bigcup_{\alpha\in\lambda}d_\alpha$ is a full metric on $X$, extending the graph metric $d$, and witnessing that the strategy $\Stra$ of the player I in the metric-extending game $\Game_d(\lambda,\Tau)$ is winning.

\section{Proof of Theorem~\ref{t:main2}}\label{s:main2}

Let $d$ be a graph metric on a set $X$ such that the subspace $[X]^2\setminus E_d $ of the metric space $([X]^2,\ddot d)$ is not discrete, and let $(F_{xy})_{xy\in [X]^2\setminus E_d }$ be any indexed family of non-degenerated subsets of the real line. Since $[X]^2\setminus E_d $ is not discrete, there exists a doubleton $p\in [X]^2\setminus  E_d $ such that for every $\e\in\IR_+$ there exists a doubleton $q\in[X]^2\setminus  E_d $ with $q\ne p$ and $\ddot d(p,q)<\e$. In particular, there exists a doubleton $q\in [X]^2\setminus E_d $ such that $q\ne p$ and $\ddot d(p,q)<\frac13\diam(F_p)$, where $\diam(F_p)\defeq\sup\{|x-y|:x,y\in F_p\}$. Choose any real number $r_q\in F_q$. Assuming that $|x-r_q|\le \ddot d(p,q)$ for all $x\in F_p$, we conclude that $|x-y|\le|x-r_q|+|r_q-y|\le 2\ddot d(p,q)<\frac 23\diam(F_p)$ for all $x,y\in F_p$ and hence $0<\diam(F_p)\le\frac23\diam(F_p)$, which is a contradiction showing that $|r_p-q_q|>\ddot d(p,q)$ for some real number $r_p\in F_p$.

Choose any function $r\in \prod_{xy\in [X]^2\setminus  E_d }F_{xy}$ such that $r(p)=r_p$ and $r(q)=r_q$. Assuming that the set $D\defeq d\cup\{\langle xy,r(xy)\rangle:xy\in [X]^2\setminus[d]\}$ is a full metric on $X$, we can apply Lemma~\ref{l:double} and conclude that $$|r_p-r_q|=|D(p)-D(q)|=|\hat D(p)-\hat D(q)|\le\ddot D(p,q)\le \ddot d(p,q),$$
which contradicts the choice of $r_p$.

\section{Proof of Theorem~\ref{t:main3}}\label{s:main3}
 
 Let $d$ be a graph metric on a set $X$ such that the metric space $([X]^2,\ddot d)$ contains a separable uncountable subspace $S\subseteq [X]^2\setminus E_d $. We need to prove that for any ordinal $\lambda$ and any family $\F$ of nondegenerated subsets of the real line,  the player II has a winning strategy in the metric-extending game $\Game_d(\lambda,\F)$. 
 First we give a precise definition of a (winning) strategy of the player II in the metric-extending game.

A {\em strategy} of the player II in the game $\Game_d(\lambda,\F)$ is a function $\Strat:\dom[\Strat]\to\IR$, defined on the set $\dom[\Strat]\defeq\bigcup_{\beta\in\lambda}([X]^2\times\F)^{\beta+1}$. Every element $p\in\dom[\Strat]$ is a function $p:\beta+1\to [X]^2\times\F$, which determines unique functions $p_1:\beta+1\to[X]^2$ and $p_2:\beta+1\to\F$ such that $p(\alpha)=\langle p_1(\alpha),p_2(\alpha)\rangle$ for every $\alpha\in\beta+1$.

A strategy $\Strat$ of the player II is {\em winning} if for every function $p:\lambda\to [X]^2\times \F$  the set $D\defeq d\cup\{\langle p_1(\alpha),\Strat(p{\restriction}_{\alpha+1})\rangle:\alpha\in\kappa\}$ is not a full metric on $X$.

It remains to describe a winning strategy of the player II in the game $\Game_d(\lambda,\F)$. In fact, the strategy is very simple: just wait for a convenient moment for destroying the triangle inequality. A formal definition of this strategy is as follows.

Inductively we shall construct a transfinite sequence $(\Strat_\beta)_{\beta\in\lambda}$ of functions $\Strat_\beta:([X]^2\times\F)^{\beta+1}\to\IR$ such that for every ordinal $\beta\in\lambda$ and every $p\in([X]^2\times\F)^{\beta+1}$, the following two conditions hold:
\begin{enumerate}
\item $\Strat_\beta(p)\in p_2(\beta)$, and
\item if $p_1(\beta)\in S$ and the set $$A_{p}\defeq\big\{\alpha\in\beta: p_1(\alpha)\in S\;\wedge\;\exists x\in p_2(\beta)\;\; \ddot d(p_1(\alpha),p_1(\beta))<|\Strat_\alpha(p{\restriction}_{\alpha+1})-x|\big\}$$ is not empty, then $\ddot d(p_1(\alpha),p_1(\beta))<|\Strat_\alpha(p{\restriction}_{\alpha+1})-\Strat_\beta(p)|$ for the ordinal $\alpha\defeq\min A_p$.
\end{enumerate}
Let us show that the function $\Strat\defeq\bigcup_{\beta\in\lambda}\Strat_\beta$ is a winning strategy of the player II in the game $\Game_d(\lambda,\F)$. To derive a contradiction, assume that this strategy is not winning. Then there exists a function $p:\lambda\to [X]^2\times\F$ such that the function $D\defeq d\cup \{\langle p(\alpha),\Strat_\alpha(p{\restriction}_{\alpha+1})\rangle:\alpha\in\lambda\}$ is a full metric on $X$. Since $\dom[D]=[X]^2$, the set $p_1^{-1}[S]\subseteq\lambda$ is uncountable. Then there exists an increasing function $\mu:\w_1\to \lambda$ such that $\mu[\w_1]\subseteq p^{-1}[S]$. Choose a function $\pi:\omega_1\to[\IR]^2$ such that $\Strat_\alpha(p{\restriction}_{\mu(\alpha)+1})\in \pi(\alpha)\subseteq p_2(\mu(\alpha))$ for every $\alpha\in \w_1$. Consider the function $$\eta:\w_1\to S\times [\IR]^2,\quad \eta:\alpha\mapsto \langle p_1\circ\mu(\alpha),\pi(\alpha)\rangle.$$ 

Let $e:[\IR]^2\to\IR$, $e:st\mapsto |s-t|$, be the Euclidean full metric  on $\IR$. This full metric induces the metric $$\ddot e:[\IR]^2\times[\IR]^2\to\IR,\quad \ddot e(xy,st)=\min\{|x-s|+|y-t|,|x-t|+|y-s|\}.$$ It is clear that $([\IR]^2,\ddot e)$ is a  separable metric spaces. On the other hand, $(S,\ddot d{\restriction}_{S\times S})$ is a separable metric space. Then the product $S\times[\IR]^2$ is a separable metric space. By \cite[4.1.15]{Eng}, every subspace of a separable metric space is Lindel\"of. In particular, the subspace $\eta[\w_1]$ of $S\times[\IR]^2$ is Lindel\"of, which implies the existence of an ordinal $\alpha\in\w_1$ such that for every neighborhood $U$ of $\eta(\alpha)$ in $S\times [\IR]^2$, the set $\eta^{-1}[U]$ is uncountable. In particular, for the neighborhood $$U\defeq\{\langle x,y\rangle\in S\times[\IR]^2:\ddot d(x,p_1\circ \mu(\alpha))+\ddot e(y,\pi(\alpha))<\tfrac12\diam(\pi(\alpha))\},$$ the preimage $\eta^{-1}[U]\subseteq\w_1$ is uncountable and hence contains some countable ordinal $\beta>\alpha$. Consider the ordinals $\alpha'\defeq\mu(\alpha)$ and $\beta'\defeq\mu(\beta)$ in $\lambda$. The inclusion $\eta(\beta)\in U$ implies $\ddot d(p_1(\beta'),p_1(\alpha'))+\ddot e(\pi(\beta),\pi(\alpha))<\frac12\diam(\pi(\alpha))$. Consider the real number $x\defeq \Strat_{\alpha'}(p{\restriction}_{\alpha'+1})$.  The definition of $\pi$ ensures that $x\in \pi(\alpha)\subseteq p_2(\alpha')$. Let $y$ be the unique point of the set $\pi(\alpha)\setminus\{x\}$. It follows from $\ddot e(\pi(\beta),\pi(\alpha))<\tfrac12\diam (\pi(\alpha))$ that there exists a point $y'\in \pi(\beta)$ such that $|y'-y|\le\ddot e(\pi(\beta),\pi(\alpha))<\frac12\diam(\pi(\alpha))=\frac12|x-y|$. Then $$|x-y'|\ge|x-y|-|y'-y|>\tfrac12\diam(\pi(\alpha))>\ddot d(p_1(\alpha'),p_1(\beta'))$$ and hence $\alpha'\in A_{p{\restriction}_{\beta'}}$. In this case, the inductive condition (2) and Lemmas~\ref{l:double} and \ref{cl:1} ensure that 
$$
\begin{aligned}
\ddot d(p_1(\alpha''),p_1(\beta'))&<|\Strat_{\alpha''}(p{\restriction}_{\alpha''+1})-\Strat_{\beta'}(p{\restriction}_{\beta'+1})|=|D(p_1(\alpha''))-D(p_1(\beta')|\\
&=|\hat D(p_1(\alpha''))-\hat D(p_1(\beta')|\le \ddot D(p_1(\alpha''),p_1(\beta'))\le\ddot d(p_1(\alpha''),p_1(\beta')),
\end{aligned}
$$which is a desired contradiction witnessing that $D$ is not a full metric. Therefore, the strategy $\Strat$ of the player I in the game $\Game_d(\lambda,\F)$ is winning.

\section{Proof of Example~\ref{ex:glue}}\label{s:glue}

 Let $p$ be a graph pseudometric and $\mathcal F$ be a family of graph pseudometrics satisfying the following conditions:
 \begin{enumerate}
\item for every $f\in\F$ the intersection $V_p\cap V_f$ is not empty;
\item for every $f\in\F$ and $xy\in[V_f\cap V_p]^2$, we have $\hat p(x,y)=\hat g(x,y)$;
\item for any distinct graph pseudometrics $f,g\in\mathcal F$, we have $V_f\cap V_g\subseteq V_p$.
\end{enumerate}

Since every $f\in\{p\}\cup\F$ is a graph pseudometric, $f(xy)=\hat f(xy)$ for every $xy\in E_f$. Now the property (2) implies that $f(xy)=\hat f(xy)=\hat p(xy)=p(xy)$  for all $xy\in E_f\cap E_p$, and  the property (3) implies that for any distinct $f,g\in\{p\}\cup\F$ and every $xy\in E_f\cap E_g$, we have $xy\subseteq V_f\cap V_g\subseteq V_p$ and hence $f(xy)=\hat f(xy)=\hat p(xy)=\hat g(xy)=g(xy)$. Consequently, the set $d\defeq p\cup\bigcup\F$ is a function defined on the graph $E_d=E_p\cup\bigcup_{f\in\F}E_f$ whose set of vertices $V_d$ is equal to $\bigcup E_d=V_p\cup\bigcup_{f\in\F}V_f$. Since the graphs $E_p$ and $E_f$, $f\in\F$, are connected and $V_f\cap V_p\ne\emptyset$ for all $f\in\F$, the graph $E_d$ is connected, too. 

Our first aim is deducing the formula
$$\hat d(x,y)=\begin{cases}\hat f(x,y)&\mbox{if $x,y\in V_f$ for some $f\in\{p\}\cup \F$};\\
{\displaystyle\inf_{\substack{a\in V_f\cap V_p\\ b\in V_g\cap V_p}}\hat f(x,a)+\hat p(a,b)+\hat g(b,y)}&\mbox{if $x\in V_f$ and $y\in V_g$ for distinct $f,g\in \{p\}\cup\F$};
\end{cases}
$$
 for calculating the shortest-path pseudometric $\hat d:V_d\times V_d\to\bar\IR_+$, induced by the function $d$. This formula will be proved in Lemma~\ref{l:formula} preceded by Claim~\ref{cla:1}--\ref{cla:8}. The graph pseudometric property of the function $d$ will be proved in Lemma~\ref{l:gp} and conditions under which the graph pseudometric is floppy will be given in Lemma~\ref{l:floppy}.

 \begin{claim}\label{cla:1} For every $n\in\IN$ and every vertices $x_0,x_1,\dots,x_n\in V_d$ with $x_0,x_n\in V_p$ and $\{x_{i-1}x_i:0<i\le n\}\subseteq E_d$, we have $\hat p(x_0,x_n)\le\sum_{i=1}^nd(x_{i-1}x_i)$.
 \end{claim}
 
\begin{proof} This claim will be proved by induction on $n$. For $n=1$ it is true as $p\subseteq d$. Assume that for some integer $n\ge 2$ we have proved that for every positive $k<n$ and points $x_0,\dots,x_k\in V_d$ with $x_0,x_k\in V_p$ and $\{x_{i-1}x_i:0<i\le k\}\subseteq E_d$ we have $\hat p(x_0,x_k)\le\sum_{i=1}^kd(x_{i-1}x_i)$.

Take any sequence $x_0,\dots,x_n\in V_d$ with $x_0,x_n\in V_p$ and $\{x_{i-1}x_i:0<i\le n\}\subseteq E_d$. If  $\{x_i:0\le i\le n\}\subseteq V_p$, then $$\hat p(x_0x_n)\le \sum_{i=1}^n\hat p(x_{i-1},x_i)\le\sum_{i=1}^np(x_{i-1}x_i)=\sum_{i=1}^nd(x_{i-1}x_i)$$and we are done. So, assume that  $\{x_i:0\le i\le n\}\not\subseteq V_p$ and  let $m\in\{0,\dots,n\}$ be the largest number such that $x_m\notin V_p$. It follows from $x_n\in V_p$ that $m<n$. The maximality of $m$ ensures that $x_{m+1}\in V_p$. Let $k\in\{1,\dots,m\}$ be the smallest number such that $\{x_i:k\le i\le m\}\cap V_p=\emptyset$. The minimality of $k$ ensures that $x_{k-1}\in V_p$. Since $x_{k-1}x_{k}\in E_d\setminus E_p\subseteq\bigcup_{f\in\F}E_f$, there exists $f\in\F$ such that $x_{k-1}x_{k}\in E_f$.
By induction we shall prove that $x_{i-1}x_{i}\in E_f$ for all $i\in\{k,\dots,m+1\}$. For $i=k$ the inclusion $x_{i-1}x_i\in E_f$ follows from the choice of $E_f$. Assume that for some $i\in\{k,\dots,m\}$ we have proved that $x_{i-1}x_i\in E_f$. Assuming that $x_ix_{i+1}\notin E_f$ and taking into account that $x_i\notin V_p$, we conclude that $x_ix_{i+1}\in E_g$ for some $g\in\F\setminus\{f\}$. The property (3) of the family $\F$ ensures that $x_i\in x_{i-1}x_i\cap x_ix_{i+1}\in V_f\cap V_g\subseteq V_p$, which is a contradiction showing that $x_ix_{i+1}\in E_f$. The triangle inequality for the full pseudometric $\hat f$ ensures that
$$\hat f(x_{k-1},x_{m+1})\le \sum_{i=k}^{m+1}\hat f(x_{i-1},x_i)\le\sum_{i=k}^{m+1}f(x_{i-1}x_i)=\sum_{i=k}^{m+1}d(x_{i-1}x_i).$$ By the property (2) of the family $\F$, $\hat p(x_{k-1},x_{m+1})=\hat f(x_{k-1},x_{m+1})$.

On the other hand, the triangle inequality for the  pseudometric $\hat p$ ensures that
$$
\begin{aligned}
\hat p(x_0x_n)&\le \hat p(x_0x_{k-1})+\hat p(x_{k-1}x_{m+1})+\hat p(x_{m+1},x_n)\\
&=\hat p(x_0x_{k-1})+\hat f(x_{k-1}x_{m+1})+\hat p(x_{m+1},x_n)\\
&\le
\sum_{i=1}^{k-1}d(x_{i-1}x_i)+\sum_{i=k}^{m+1}d(x_{i-1}x_i)+\sum_{i=m+2}^nd(x_{i-1}x_i)=\sum_{i=1}^nd(x_{i-1}x_i).
\end{aligned}
$$
by the inductive assumption, applied to the chains $x_0,\dots, x_{k-1}$ and $x_{m+1},\dots,x_n$. 
\end{proof}

\begin{claim}\label{cla:2} For every $x,y\in V_p$ we have $\hat d(x,y)=\hat p(x,y)$.
\end{claim}

\begin{proof} Claim~\ref{cla:1} and the definition of the pseudometric $\hat d$ ensure that $\hat p(x,y)\le \hat d(x,y)$. Assuming that $\hat p(x,y)<\hat d(x,y)$, we can find a sequence of vertices $x_0,\dots,x_n\in V_p$ such that $x_0x_n=xy$, $\{x_{i-1}x_i:0<i\le n\}\subseteq E_p$ and $\sum_{i=1}^np(x_{i-1}x_i)<\hat d(x,y)$. It follows from $p\subseteq d$  and $\{x_{i-1}x_i:0<i\le n\}\subseteq E_p\subseteq E_d$ that $$\hat d(x,y)\le \sum_{i=1}^nd(x_{i-1}x_i)=\sum_{i=1}^n p(x_{i-1}x_i)<\hat d(x,y),$$which is a contradiction showing that $\hat p(x,y)=\hat d(x,y)$.
\end{proof}

\begin{claim}\label{cla:3} For every $f\in\F$ and $x,y\in V_p\cap V_f$ we have $\hat d(x,y)=\hat p(x,y)=\hat f(x,y)$.
\end{claim}

\begin{proof} The equality $\hat d(x,y)=\hat p(x,y)$ follows from Claim~\ref{cla:2} and $\hat p(x,y)=\hat f(x,y)$ follows from the property (2) of the family $\F$. 
\end{proof}

\begin{claim}\label{cla:4} For every $f\in\F$ and vertices $x\in V_f\setminus V_p$ and $y\in V_f\cap V_p$, we have $\hat d(x,y)=\hat f(x,y)$.
\end{claim}

\begin{proof} It follows from $f\subseteq d$ that $\hat d(x,y)\le\hat f(x,y)$. Assuming that $\hat d(x,y)<\hat f(x,y)$, we can find vertices $x_0,\dots,x_n\in V_d$ such that $x_0=x$, $x_n=y$, $\{x_{i-1}x_i:0<i\le n\}\subseteq E_d$ and $\sum_{i=1}^nd(x_{i-1}x_i)<\hat f(x,y)$.

 Let $k\in\{0,\dots,n\}$ be the largest number such that $\{x_i:0\le i<k\}\subseteq V_f\setminus V_p$. The maximality of $k$ and $x_n=y\in V_p$ imply $x_{k}\notin V_f\setminus V_p$. We claim that $\{x_{i-1}x_i:0<i\le k\}\subseteq E_f$. In the opposite case, we can find $i\in\{1,\dots,k\}$  and $g\in(\{p\}\cup\F)\setminus\{f\}$ such that $x_{i-1}x_i\in E_g$ and hence $x_{i-1}x_i\subseteq V_g$. The property (3) of the family $\F$ ensures that $x_{i-1}x_i\subseteq V_ f\cap V_g\subseteq V_p$. If $i<k$, then $x_i\in x_{i-1}x_i\subseteq V_p$ contradicts the choice of $k$. If $i=k$, then $x_{k-1}\in x_{i-1}x_i\subseteq V_p$ contradicts the choice of $k$. In both cases we obtain a contradiction showing that $\{x_{i-1}x_i:0<i\le k\}\subseteq E_f$ and hence $x_{k}\in V_f\setminus(V_f\setminus V_p)=V_f\cap V_p$. 

The triangle inequality for the pseudometric $\hat f$ and Claim~\ref{cla:3} imply a contradiction
$$
\begin{aligned}
\hat f(x,y)&=\hat f(x_0,x_n)\le \hat f(x_0,x_{k})+\hat f(x_{k},x_n)=\hat f(x_0,x_k)+\hat d(x_k,x_n)\\
&\le\sum_{i=1}^{k} f(x_{i-1}x_i)+\sum_{i=k+1}^n d(x_{i-1}x_i)
=\sum_{i=1}^{k} d(x_{i-1}x_i)+\sum_{i=k+1}^n d(x_{i-1}x_i)\\
&=\sum_{i=1}^nd(x_{i-1}x_i)<\hat f(x,y),
\end{aligned}
$$
witnessing that $\hat d(x,y)=\hat f(x,y)$.
\end{proof}

\begin{claim}\label{cla:5} For every $f\in\F$ and vertices $x,y\in V_f\setminus V_p$, we have $\hat d(x,y)=\hat f(x,y)$.
\end{claim}

\begin{proof} It follows from $f\subseteq d$ that $\hat d(x,y)\le\hat f(x,y)$. Assuming that $\hat d(x,y)<\hat f(x,y)$, we can find a sequence $x_0,\dots,x_n\in V_d$ such that $x_0=x$, $x_n=y$, $\{x_{i-1}x_i:0<i\le n\}\subseteq E_d$ and $\sum_{i=1}^nd(x_{i-1}x_i)<\hat f(x,y)$. Assuming that $\{x_{i-1}x_i:0<i\le n\}\subseteq E_f$, we conclude that
$$\hat f(x,y)\le \sum_{i=1}^nf(x_{i-1}x_i)=\sum_{i=1}^nd(x_{i-1}x_i)<\hat f(x,y),$$
which is a contradiction showing that $\{x_{i-1}x_i:0<i\le n\}\not\subseteq E_f$. We claim that $\{x_i:0\le i\le n\}\not\subseteq V_f\setminus V_p$. To derive a contradiction, assume that $\{x_i:0\le i\le n\}\subseteq V_f\setminus V_p$. Since $\{x_{i-1}x_i:0<i\le n\}\not\subseteq E_f$, there exists $i\in\{1,\dots,n\}$ such that $x_{i-1}x_i\notin E_f$ and hence $x_{i-1}x_i\in E_g$ for some $g\in(\{p\}\cup\F)\setminus\{f\}$. The property (3) of the family $\F$ ensures that $x_ix_{i-1}\in V_f\cap V_g\subseteq V_p$, which contradicts our assumption $x_{i-1}x_i\subseteq V_f\setminus V_p$. This contradiction shows that $x_i\notin V_f\setminus V_p$ for some  $i\in\{0,\dots,n\}$. Let $k\in\{0,\dots,n\}$ be the largest number such that $\{x_i:0\le i<k\}\subseteq V_f\setminus V_p$ and $m\in\{0,\dots,n\}$ be the minimal number such that $\{x_i:m<i\le n\}\subseteq V_f\setminus V_p$. Since $x_0x_n=xy\subseteq V_f\setminus V_p$ and $\{x_i:0<i\le n\}\not\subseteq V_f\setminus V_p$, the numbers $k$ and $m$ are well-defined and $k\le m$. Moreover, the maximality of $k$ and the minimality of $m$ ensure that $x_{k},x_{m}\notin V_f\setminus V_p$. 

We claim that $x_{k-1}x_{k}\in E_f$. In the opposite case, $x_{k-1}x_{k}\in E_g$ for some $g\in(\{p\}\cup\F)\setminus\{f\}$ and hence $x_{k-1}\in (V_f\setminus V_p)\cap V_g=\emptyset$ by the property (3) of the family $\F$. This contradiction shows that $x_{k-1}x_{k}\in E_f$ and hence $x_{k-1}x_{k}\in V_f$ and $x_{k}\in V_f\setminus(V_f\setminus V_p)=V_f\cap V_p$. By analogy we can prove that $x_{m}\in V_f\cap V_p$. Now the triangle inequality for the pseudometric $\hat f$ and Claims~\ref{cla:3} and \ref{cla:4} ensure that
$$
\begin{aligned}
\hat f(x,y)&=\hat f(x_0,x_n)\le \hat f(x_0,x_{k})+\hat f(x_{k},x_{m})+\hat f(x_{m},x_n)=\hat d(x_0,x_{k})+\hat d(x_{k},x_{m})+\hat d(x_{m},x_n)\\
&\le\sum_{i=1}^{k}d(x_{i-1}x_i)+\sum_{i=k+1}^{m}d(x_{i-1}x_i)+\sum_{i=m+1}^nd(x_{i-1}x_i)=\sum_{i=1}^nd(x_{i-1}x_i)<\hat f(x,y),
\end{aligned}
$$which is a contradiction showing that $\hat d(x,y)=\hat f(x,y)$.
\end{proof}

\begin{claim}\label{cla:6} For every $f\in\{p\}\cup\F$ and points $x,y\in V_f$ we have $\hat d(x,y)=\hat f(x,y)$.
\end{claim}

\begin{proof} If $f=p$, then the equality $\hat d(x,y)=\hat p(x,y)=\hat f(x,y)$ was proved in Claim~\ref{cla:2}. So, we assume that $f\ne p$ and hence $f\in\F$.

 If $x,y\in V_p$, then the equality $\hat d(x,y)=\hat p(x,y)=\hat f(x,y)$ has been proved in Claim~\ref{cla:3}.

If $x\notin V_p$ and $y\in V_f$, then the equality $\hat d(x,y)=\hat f(x,y)$ was proved in Claim~\ref{cla:4}.

If $x\in V_p$ and $y\notin V_f$, then the equality $\hat d(x,y)=\hat d(y,x)=\hat f(y,x)=\hat f(x,y)$ follows from Claim~\ref{cla:4} and the symmetry of the pseudometrics $\hat d$ and $\hat f$.

If $x,y\notin V_p$, then the equality $\hat d(x,y)=\hat f(x,y)$ was proved in Claim~\ref{cla:5}.
\end{proof}

\begin{claim}\label{cla:7} For every $x\in V_p$, $g\in \F$ and $y\in V_g$, we have $$\hat d(x,y)=\delta\defeq\inf\{\hat p(x,b)+\hat g(b,y):b\in V_g\cap V_p\}.$$
\end{claim}

\begin{proof} Assuming that $\delta<\hat d(x,y)$, we can find a point $b\in V_g\cap V_p$ such that $\hat p(x,b)+\hat g(b,y)<\hat d(x,y)$. Applying Claims~\ref{cla:2} and \ref{cla:3}, we conclude that
$$\hat p(x,b)+\hat g(b,y)<\hat d(x,y)\le \hat d(x,b)+\hat d(b,y)=\hat p(x,b)+\hat g(b,y),$$which is a desired contradiction showing that $\hat d(x,y)\le \delta$. If $y\in V_p$, then by Claim~\ref{cla:2},
$$\delta\le \hat p(x,y)+\hat g(y,y)=\hat d(x,y)+0=\hat d(x,y)$$and hence $\hat d(x,y)=\delta$.

So, assume that $y\notin V_p$. Assuming that $\hat d(x,y)<\delta$, we can find a sequence $x_0,\dots,x_n\in V_d$ such that $x_0=x$, $x_n=y$, $\{x_{i-1}x_i:0<i\le n\}\subseteq E_d$ and $\sum_{i=1}^n d(x_{i-1}x_i)<\delta$. Let $k\in\{1,\dots,n\}$ be the smallest number such that $\{x_i:k< i\le n\}\subseteq V_g\setminus V_p$. Since $x_0\in V_p$, the minimality of $k$ ensures that $x_{k}\notin V_g\setminus V_p$. Since $x_{k}x_{k+1}\in E_d$, there exists a graph pseudometric $f\in\{p\}\cup \F$ such that $x_{k}x_{k+1}\in E_f$ and hence $x_{k}x_{k+1}\subseteq V_f$. Assuming that $f\ne g$ and applying the property (3) of the family $\F$, we conclude that $x_{k+1}\in V_f\cap (V_g\setminus V_p)=\emptyset$, which is a contradiction showing that $f=g$ and hence $x_{k}\in V_f\setminus(V_g\setminus V_p)=V_g\cap V_p$. Then for the point $b=x_{k}\in V_g\cap V_p$, we can apply Claims~\ref{cla:2}, \ref{cla:3} and obtain a contradition
$$
\begin{aligned}
\hat p(x,b)+\hat g(b,y)&=\hat p(x_0,x_{k})+\hat g(x_{k},x_n)=\hat d(x_0,x_{k})+\hat d(x_{k},x_n)\\
&\le \sum_{i=1}^{k}d(x_{i-1}x_i)+\sum_{i=k+1}^nd(x_{i-1}x_i)=\sum_{i=1}^n d(x_{i-1}x_i)<\delta\le \hat p(x,b)+\hat g(b,y),
\end{aligned}$$
showing that $\hat d(x,y)=\delta$.
\end{proof}

\begin{claim}\label{cla:8} For every distinct elements $f,g\in\F$ and points  $x\in V_f$, $y\in V_g$, we have $$\hat d(x,y)=\inf\{\hat f(x,a)+\hat p(a,b)+\hat g(b,y):a\in V_f\cap V_p,\;b\in V_p\cap V_g\}.$$
\end{claim}

\begin{proof} Let $\delta\defeq \inf \{\hat f(x,a)+\hat p(a,b)+\hat g(b,y):a\in V_f\cap V_p,\;b\in V_p\cap V_g\}$.

If $x\in V_f\cap V_p$, then for every $b\in V_p$ we have
$$\hat p(x,b)=\hat p(x,x)+\hat p(x,b)=\inf\{\hat p(x,a)+\hat p(a,b):a\in V_f\cap V_p\}=\inf\{\hat f(x,a)+\hat p(a,b):a\in V_f\cap V_a\}$$by the traingle inequality for the pseudometric $\hat p$ and Claim~\ref{cla:3}. By Claim~\ref{cla:7},
$$
\begin{aligned}
\hat d(x,y)&=\inf\{\hat p(x,b)+\hat g(b,y):b\in V_p\cap V_g\}\\
&=\inf\{\inf\{\hat f(x,a)+\hat p(a,b):a\in V_f\cap V_p\}+\hat g(b,y):b\in V_p\cap V_g\}\\
&=\inf\{\hat f(x,a)+\hat p(a,b)+\hat g(b,y):a\in V_f\cap V_p,\;b\in V_p\cap V_g\}=\delta.
\end{aligned}
$$
By analogy we can prove that $\hat d(x,y)=\delta$ if $y\in V_p$. 

So, assume that $x\in V_f\setminus V_p$ and $y\in V_g\setminus V_p$. Assuming that $\delta<\hat d(x,y)$, by the definition of $\delta$, there exist points $a\in V_f\cap V_p$ and $b\in V_p\cap V_g$ such that $\hat f(x,a)+\hat p(a,b)+\hat g(b,y)<\hat d(x,y)$. Applying the triangle inequality for the pseudometric $\hat d$ and Claims~\ref{cla:2}, \ref{cla:3}, we obtain a contradiction
$$
\hat f(x,a)+\hat p(a,b)+\hat g(b,y)<\hat d(x,y)\le \hat d(x,a)+\hat d(a,b)+\hat d(b,y)=\hat f(x,a)+\hat p(a,b)+\hat g(b,y),$$
showing that $\hat d(x,y)\le \delta$. 

Assuming that $\hat d(x,y)<\delta$, we can find a sequence $x_0,\dots,x_n\in V_d$ such that $x_0=x$, $x_n=y$, $\{x_{i-1}x_i:0<i\le n\}\subseteq E_d$ and $\sum_{i=1}^n d(x_{i-1}x_i)<\delta$. The property (3) of the family $\F$ ensures that $V_f\cap V_g\subseteq V_p$ and hence $x\in V_f\setminus V_p\subseteq V_f\setminus V_g$ and $y\in V_g\setminus V_p\subseteq V_g\setminus V_f$. Let $k\in\{0,\dots,n\}$ be the largest number such that $\{x_i:0\le i<k\}\subseteq V_f\setminus V_p$ and let $m\in\{0,\dots,n\}$ be the smallest number such that $\{x_i:m<i\le n\}\subseteq V_g\setminus V_p$. Since $x_0=x\notin V_g$ and $x_n=y\notin V_f$, the numbers $k$ and $m$ are well-defined. Repeating the argument from the proof of Claim~\ref{cla:7}, we can show that  $x_k\in V_f\cap V_p$ and $x_m\in V_p\cap V_g$.
Applying Claims~\ref{cla:2} and \ref{cla:3}, we obtain a contradiction
$$\begin{aligned}
\delta&\le \hat f(x,x_k)+\hat p(x_k,x_m)+\hat g(x_m,x_n)= \hat d(x,x_k)+\hat d(x_k,x_m)+\hat d(x_m,x_n)\\
&\le\sum_{i=1}^k d(x_{i-1}x_i)+\sum_{i=k+1}^m d(x_{i-1}x_i)+\sum_{i=m+1}^n d(x_{i-1}x_i)=\sum_{i=1}^nd(x_{i-1}x_i)<\delta,
\end{aligned}
$$
showing that $\hat d(x,y)=\delta$.
\end{proof}

\begin{lemma}\label{l:formula} For every $x,y\in V_d$, the shortest-path distance $\hat d(x,y)$ can be calculated by the formula
$$\hat d(x,y)=\begin{cases}\hat f(x,y)&\mbox{if $x,y\in V_f$ for some $f\in\{p\}\cup \F$};\\
{\displaystyle\inf_{\substack{a\in V_f\cap V_p\\ b\in V_g\cap V_p}}\hat f(x,a)+\hat p(a,b)+\hat g(b,y)}&\mbox{if $x\in V_f$ and $y\in V_g$ for distinct $f,g\in \{p\}\cup\F$}.
\end{cases}
$$
\end{lemma}

\begin{proof} If $x,y\in V_f$ for some $f\in\{p\}\cup\F$, then $\hat d(x,y)=\hat f(x,y)$ by Claims~\ref{cla:2} and \ref{cla:3}.

Next, assume that $x\in V_f$ and $y\in V_g$ for distinct $f,g\in\{p\}\cup\F$. 

If $f=p$, then $g\in (\{p\}\cup\F)\setminus\{f\}\subseteq \F$. Applying Claim~\ref{cla:7}, we conclude that 
$$
\begin{aligned}
\hat d(x,y)&=\inf\{\hat p(x,b)+\hat g(b,y):b\in V_g\cap V_p\}\\
&=\inf\{\inf\{\hat p(x,a)+\hat p(a,b):a\in  V_f\cap V_p\}+\hat g(b,y):b\in V_g\cap V_p\}\\
&=\inf\{\hat f(x,a)+\hat p(a,b)+\hat g(b,y):a\in V_f\cap V_p,\; b\in V_p\cap V_g\}.
\end{aligned}
$$
By analogy we can prove that $\hat d(x,y)=\inf\{\hat f(x,a)+\hat p(a,b)+\hat g(b,y):a\in V_f\cap V_p,\; b\in V_p\cap V_g\}$ if $g=p$.

If $f\ne p\ne g$, then $f,g\in\F$ and then $$\hat d(x,y)=\inf\{\hat f(x,a)+\hat p(a,b)+\hat g(b,y):a\in V_f\cap V_p,\; b\in V_p\cap V_g\}$$by Claim~\ref{cla:8}.
\end{proof}

\begin{lemma}\label{l:gp} The function $d$ is a graph pseudometric.
\end{lemma}

\begin{proof} Given any points $x,y\in V_d$ with $xy\in E_d=E_p\cup\bigcup_{f\in \F}E_f$, we need to check that $d(xy)=\hat d(x,y)$. Find $f\in\{p\}\cup\F$ such that $xy\in E_f$. Applying Claims~\ref{cl:2} and \ref{cl:3} and taking into account that $f$ is a graph pseudometric, we conclude that
$$d(xy)=f(xy)=\hat f(xy)=\hat d(x,y).$$
\end{proof}

Now we give conditions under which the graph pseudometric $d$ is floppy.

For every vertex $v\in V_d$ and nonempty subset $B\subseteq V_d$, consider the real number $$\Lambda(v;B)\defeq \inf\{\hat d(a,x)+\hat d(x,b)-\hat d(a,b):a,b,\in B\}.$$ 

\begin{claim}\label{cla:flap1} For every $f\in V_f$, $x\in V_f\setminus V_p$ and $y\in V_p\setminus V_f$, we have $$\hat d(x,y)-\check d(x,y)\ge \delta\defeq\min\{\Lambda(x;V_f\cap V_p),\tfrac12\Lambda(y;V_p\cap V_f)\}.$$
\end{claim}

\begin{proof} To derive a contradiction, assume that $\hat d(x,y)-\check d(x,y)<\delta$.
Since $\hat d(x,y)<\check d(x,y)+\delta$, by Claims~\ref{cla:7} and \ref{cla:2}, \ref{cla:3}, there exists a vertex $a\in V_f\cap V_p$ such that $$\hat d(x,a)+\hat d(a,y)<\check d(x,y)+\delta.$$

 If $\check d(x,y)=0$, then $\hat d(x,a)\le\hat d(x,a)+\hat d(a,y)<\delta$. Applying Claim~\ref{cla:3}, we obtain a desired contradiction: 
$
\delta>\hat d(x,a)=\tfrac 12(\hat d(a,x)+\hat d(x,a)-\hat d(a,a))\ge \tfrac12\Lambda(x;V_f\cap V_p)\ge\delta.
$

So, we assume that $\check d(x,y)>0$. In this case, the strict inequality $\check d(x,y)>\hat d(x,a)+\hat d(a,y)-\delta$ implies the existence of vertices $u,v\in V_d$ such that $uv\in E_d$ and $$\hat d(u,v)-\hat d(u,x)-\hat d(v,y)>\hat d(x,a)+\hat d(a,y)-\delta.$$ Since $uv\in E_d=\bigcup_{g\in\{p\}\cup\F}E_g$, there exists $g\in\{p\}\cup\F$ such that $uv\in E_g$. 

Depending on the location of $g$ in the set $\{p\}\cup\{f\}\cup(\F\setminus\{p,f\})$, separately we consider three cases.

1. First assume that $g=p$ and hence $uv\subseteq V_p$. In this case we obtain a contradiction as
$$
\begin{aligned}
\hat d(v,a)+\hat d(a,u)&\ge \hat d(v,u)=\hat d(u,v)\\
&>\hat d(u,x)+\hat d(y,v)+\hat d(x,a)+\hat d(a,y)-\delta\\
&\ge\hat d(u,x)+\hat d(y,v)+\hat d(x,a)+\hat d(a,y)-\Lambda(x;V_f\cap V_p)\\
&\ge \hat d(u,x)+\hat d(y,v)+\hat d(x,a)+\hat d(a,y)-(\hat d(u,x)+\hat d(x,a)-\hat d(u,a))\\
&=\hat d(y,v)+\hat d(a,y)+\hat d(u,a)\ge \hat d(v,a)+\hat d(a,u).
\end{aligned}
$$
\smallskip

2. Next, assume that $g=f$ and hence $uv\subseteq V_g=V_f$. Since $$\hat d(v,y)<\hat d(u,v)-\hat d(u,x)-\hat d(x,a)-\hat d(a,y)+\delta,$$ by Claims~\ref{cla:7} and \ref{cla:2}, \ref{cla:3}, there exists an element $\alpha\in V_f\cap V_p$ such that
$$
\begin{aligned}
\hat d(v,\alpha)&+\hat d(\alpha,y)<\hat d(u,v)-\hat d(u,x)-\hat d(x,a)-\hat d(a,y)+\delta\\
&\le\hat d(u,v)-\hat d(u,x)-\hat d(x,a)-\hat d(a,y)+\Lambda(y;V_p\cap V_f)\\
&\le \hat d(u,v)-\hat d(u,x)-\hat d(x,a)-\hat d(a,y)+(\hat d(a,y)+\hat d(y,\alpha)-\hat d(a,\alpha))\\
&=\hat d(u,v)-\hat d(u,x)-\hat d(x,a)+\hat d(y,\alpha)-\hat d(a,\alpha)
\end{aligned}
$$
and we obtain a desired contradiction:
$\hat d(u,v)\le \hat d(u,x)+\hat d(x,a)+\hat d(a,\alpha)+\hat d(\alpha,v)<d(u,v)$.
\smallskip

3. Finally, assume that $g\notin\{p,f\}$. Since $\hat d(x,u)<\hat d(u,v)-\hat d(y,v)-\hat d(x,a)-\hat d(a,y)+\delta$, by Claim~\ref{cla:7}, there exists $\alpha\in V_f\cap V_p$ such that 
$$
\begin{aligned}
\hat d(x,\alpha)&+\hat d(\alpha,u)<\hat d(u,v)-\hat d(y,v)-\hat d(x,a)-\hat d(a,y)+\delta\\
&\le \hat d(u,v)-\hat d(y,v)-\hat d(x,a)-\hat d(a,y)+\Lambda(x;V_f\cap V_p)\\
&\le \hat d(u,v)-\hat d(y,v)-\hat d(x,a)-\hat d(a,y)+(\hat d(a,x)+\hat d(x,\alpha)-\hat d(a,\alpha))\\
&=\hat d(u,v)-\hat d(y,v)-\hat d(a,y)+\hat d(x,\alpha)-\hat d(a,\alpha)
\end{aligned}
$$
and we obtain a final contradiction
$
\hat d(u,v)\le \hat d(u,\alpha)+d(\alpha,a)+\hat d(a,y)+\hat d(y,v)<\hat d(u,v),
$
witnessing that $\hat d(x,y)-\check d(x,y)\ge\delta>0$.
\end{proof}

\begin{claim}\label{cla:flap2} For every distinct elements $f,g\in\F$ and points $x\in V_f\setminus V_p$ and $y\in V_g\setminus V_p$ we have $$\hat d(x,y)-\check d(x,y)\ge \delta\defeq\min\{\Lambda(x;V_f\cap V_p),\Lambda(y;V_p\cap V_f)\}.$$
\end{claim}

\begin{proof} To derive a contradiction, assume that $\hat d(x,y)-\check d(x,y)<\delta$. By Lemma~\ref{l:formula}, there exist vertices $a\in V_f\cap V_p$ and $b\in V_p\cap V_g$ such that $$\hat d(x,a)+\hat d(a,b)+\hat d(b,y)<\check d(x,y)+\delta.$$
If $\check d(x,y)=0$, then we obtain a desired contradiction as
$$
\begin{aligned}
\delta&=\min\{\Lambda(x;V_f\cap V_p),\Lambda(y;V_p\cap V_g)\}\le\min\{\hat d(a,x)+\hat d(x,a)-\hat d(a,a),\hat d(b,y)+\hat d(y,b)-\hat d(b,b)\}\\
&=2\min\{\hat d(x,a),\hat d(b,y)\}\le \hat d(x,a)+\hat d(a,b)+\hat d(b,y)<\check d(x,y)+\delta=\delta.
\end{aligned}
$$

So, we assume that $\check d(x,y)>0$. In this case the inequality $\check d(x,y)>\hat d(x,a)+\hat d(a,b)+\hat d(b,y)-\delta$ implies the existence of vertices $u,v\in V_d$ such that $uv\in E_d$ and $$\hat d(u,v)-\hat d(u,x)-\hat d(v,y)>\hat d(x,a)+\hat d(a,b)+\hat d(b,y)-\delta.$$
Since $uv\in E_d=\bigcup_{h\in\{p\}\cup\F}E_h$, there exists $h\in\{p\}\cup\F$ such that $uv\in E_h$. Since $f\ne g$, either $h\ne f$ or $h\ne g$.

First assume that $h\ne f$. Since $\hat d(x,u)<\hat d(u,v)-\hat d(v,y)-\hat d(x,a)-\hat d(a,b)-\hat d(b,y)+\delta$, by Lemma~\ref{l:formula}, there exist vertices $\alpha\in V_f\cap V_p$ and $\beta\in V_p\cap V_h$ such that 
$$\begin{aligned}
&\hat d(x,\alpha)+\hat d (\alpha,\beta)+\hat d(\beta,u) <\hat d(u,v)-\hat d(v,y)-\hat d(x,a)-\hat d(a,b)-\hat d(b,y)+\delta\\
&\le \hat d(u,v)-\hat d(v,y)-\hat d(x,a)-\hat d(a,b)-\hat d(b,y)+\Lambda(x;V_f\cap V_p)\\
&\le  \hat d(u,v)-\hat d(v,y)-\hat d(x,a)-\hat d(a,b)-\hat d(b,y)+(\hat d(a,x)+\hat d(x,\alpha)-\hat d(a,\alpha)).
\end{aligned}
$$
Then we obtain a desired contradiction as
$$
\hat d(u,v)\le \hat d(u,\beta)+\hat d(\beta,\alpha)+\hat d(\alpha,a)+\hat d(a,b)+\hat d(b,y)+\hat d(y,v)<\hat d(u,v).$$

By analogy we can derive a contradiction assuming that $h\ne  g$.
\end{proof}

\begin{lemma}\label{l:floppy} 
Assume that $p$ is a full pseudometric and every graph pseudometric $f\in\F$ is floppy and for every $x\in V_f\setminus V_p$ and $y\in V_p\setminus V_f$ the numbers $\Lambda(x;\Lambda_f\cap\Lambda_p)$ and $\Lambda(y;\Lambda_p\cap\Lambda_f)$. Then the partial pseudometric $d$ is floppy.
\end{lemma}

\begin{proof} To prove that the partial pseudometric $d$ is floppy, fix any points $x,y\in V_d$ with $xy\notin E_d$. We have to prove that $\check d(x,y)<\hat d(x,y)$.

Since $x,y\in V_d=\bigcup_{f\in\{p\}\cup\F}V_f$, there exist $f,g\in \{p\}\cup\F$ such that $x\in V_f$ and $y\in V_g$. Two cases are possible.

1. First we assume that $f\ne g$. This case has four subcases.

1.1. If $x\notin V_p$ and $y\in V_p$, then by Claim~\ref{cla:flap1},
$$\hat d(x,y)-\check d(x,y)\ge \min\{\Lambda(x;V_f\cap V_p),\tfrac12\Lambda(y;V_p\cap V_f)\}>0.$$

1.2. If $x\in V_p$ and $y\notin V_p$, then by Claim~\ref{cla:flap1},
$$\hat d(x,y)-\check d(x,y)=\hat d(y,x)-\check d(y,x)\ge \min\{\Lambda(y;V_g\cap V_p),\tfrac12\Lambda(x;V_p\cap V_g)\}>0.$$

1.3. If $x\notin V_p$ and $y\notin V_p$, then by Claim~\ref{cla:flap2},
$$\hat d(x,y)-\check d(x,y)\ge \min\{\Lambda(x;V_f\cap V_p),\Lambda(y;V_g\cap V_p)\}>0.$$

1.4. If $x\in V_p$ and $y\in V_p$, then  since $p$ is a full pseudometric, $xy\in [V_p]^2\subseteq E_p\subseteq E_d$, which contradicts the choice of $xy$. So, this case is impossible.
\smallskip

2. Next, assume that $f=g$. Since the graph pseudometric $f$ is floppy, the real number $\delta\defeq \hat f(x,y)-\check f(x,y)$ is positive.  It remains to prove that $\hat d(x,y)-\check d(x,y)\ge \delta$. If $\hat d(x,y)=0$, then by Claim~\ref{cla:3}, $\hat d(x,y)-\check d(x,y)=\hat d(x,y)=\hat f(x,y)=\delta+\check f(x,y)\ge \delta$. So, assume that $\check d(x,y)>0$. 

To derive a contradiction, assume that $\hat d(x,y)-\check d(x,y)<\delta$. 
Then $\check d(x,y)>\hat d(x,y)-\delta$ and by the definition of $\check d(x,y)>0$, there exist elements $u,v\in V_d$ such that $uv\in E_d$ and $$\hat d(u,v)-\hat d(u,x)-\hat d(v,y)>\hat d(x,y)-\delta.$$ Since $uv\in E_d=\bigcup_{h\in\{p\}\cup\F}E_h$, there exists $h\in\{p\}\cup\F$ such that $uv\in E_h$ and hence $uv\subseteq V_h$. Now consider two cases.

2.1. If $h=f$, then using Claim~\ref{cla:3}, the triangle inequality for the graph pseudometric $\hat f$ and the definition of the number $\check f(x,y)$, we obtain a contradiction as
$$
\begin{aligned}
\check f(x,y)&\ge\hat f(u,v)-\hat f(u,x)-\hat f(v,y)= \hat d(u,v)-\hat d(u,x)-\hat d(v,y)\\
&>\hat d(x,y)-\delta=\hat f(x,y)-(\hat f(x,y)-\check f(x,y))=\check f(x,y).
\end{aligned}
$$

2.2. If $h\ne f$, then by Lemma~\ref{l:formula} and the strict inequality $\hat d(u,x)+\hat d(v,y)<\hat d(u,v)-\hat d(x,y)+\delta$, there exist points $\alpha,\beta\in V_h\cap V_p$ and $a,b\in V_f\cap V_p$ such that 
$$(\hat h(u,\alpha)+\hat p(\alpha,a)+\hat f(a,x))+(\hat h(v,\beta)+\hat p(\beta,b)+\hat f(b,y))<\hat d(u,v)-\hat d(x,y)+\delta.$$
Using Claims~\ref{cla:2}, \ref{cla:3}, the triangle inequality for the pseudometric $\hat d$, and the definition of $\check f(x,y)$, we obtain a contradiction
$$\begin{aligned}
\hat f(x,y)&\ge f(ab)-\hat f(a,x)-\hat f(b,y)=d(ab)-\hat d(a,x)-\hat d(b,y)=\hat d(a,b)-\hat d(a,x)-\hat d(b,y)\\
&\ge \big(\hat d(u,v)-\hat d(u,\alpha)-\hat d(\alpha,a)-\hat d(v,\beta)-\hat d(\beta,b)\big)-\hat d(a,x)-\hat d(b,y)\\
&= \hat d(u,v)-(\hat h(u,\alpha)+\hat p(\alpha,a)+\hat f(a,x))-(\hat h(v,\beta)+\hat p(\beta,b)+\hat f(b,y))\\
&>\hat d(x,y)-\delta=\hat f(x,y)-(\hat f(x,y)-\check f(x,y))=\check f(x,y),
\end{aligned}
$$
witnessing that $\hat d(x,y)-\check d(x,y)\ge \delta>0$.
\end{proof}


\newpage

\begin{thebibliography}{}

\bibitem{Banakh} T.~Banakh, {\em Banakh spaces and their geometry}, preprint ({\tt arxiv.org/abs/2305.07354}).

\bibitem{BB} T.~Banakh, C.~Bessaga, {\em On linear operators extending [pseudo]metrics}, Bull. Polish Acad. Sci. Math. {\bf 48}:1 (2000),  35--49.

\bibitem{BSTZ} T.~Banakh, I.~Stasyuk, E.D.~Tymchatyn, M.~Zarichnyi, {\em  Extension of functions and metrics with variable domains}, Topology Appl. {\bf 231} (2017), 353--372. 

\bibitem{Bes} C.~Bessaga, {\em On linear operators and functors extending pseudometrics}, Fund. Math. {\bf 142}:2 (1993), 101--122. 

\bibitem{Bing} R.H.~Bing, {\em Extending a metric}, Duke Math. J. 14 (1947), 511--519.



\bibitem{DMV} O.~Dovgoshey, O.~Martio, M.~Vuorinen, {\em Metrization of weighted graphs}, Ann. Comb. {\bf 17}:3 (2013),  455--476.

\bibitem{Eng} R.~Engelking, {\em General Topology}, Heldermann Verlag, Berlin, 1989.

\bibitem{H} F.~Hausdorff, {\em Erweiterung einer Hom\"omorphie}, Fund. Math. {\bf 16} (1930), 353–360.

\bibitem{P} E.~Petrov, {\em On the uniqueness of continuation of a partially defined metric}, Theory Appl. Graphs {\bf 10}:1 (2023), Art. 1, 6 pp.

\bibitem{Tor} H.~Toru\'nczyk, {\em A short proof of Hausdorff's theorem on extending metrics}, Fund. Math. {\bf 77}:2 (1972),  191--193.

\bibitem{TZ} E.~Tymchatyn, M.~Zarichnyi, {\em On simultaneous linear extensions of partial (pseudo)metrics}, Proc. Amer. Math. Soc. {\bf 132}:9 (2004),  2799--2807.


\bibitem{MO} {\tt Taras Banakh}, {\em A generic metric on $X\cup\mathbb Z$}, ({\tt mathoverflow.net/q/446586/61536}).

\end{thebibliography}
\end{document}